\documentclass[reqno, a4paper]{amsart}
\usepackage[T1]{fontenc}
\usepackage[utf8]{inputenc}
\usepackage{amsmath,amsfonts,amssymb,graphics,amsthm, comment}
\usepackage{textcomp}
\usepackage{upgreek,mathtools,bm}
\usepackage{newtxtext}
\usepackage[varvw]{newtxmath}
\usepackage[colorlinks = true, linkcolor = black, urlcolor  = black, citecolor = black]{hyperref}

\usepackage{cleveref}

\usepackage{enumitem}
\usepackage{xcolor}
\usepackage{tikz-cd}

\DeclareMathOperator{\id}{id}
\DeclareMathOperator{\tr}{tr}

\DeclareMathOperator{\ri}{ri}

\newcommand{\msb}{m^{\text{SB}}}
\newcommand{\psb}{\pi^{\text{SB}}}
\newcommand{\Xsb}{X^{\text{SB}}}
\newcommand{\F}{\mathcal F}

\usepackage[a4paper, left=3cm, right=3cm]{geometry}

\theoremstyle{plain}
\newtheorem{theorem}{Theorem}

\newtheorem{lemma}[theorem]{Lemma}
\newtheorem{proposition}[theorem]{Proposition}

\theoremstyle{definition}
\newtheorem{definition}[theorem]{Definition}
\newtheorem{assumption}[theorem]{Assumption}

\theoremstyle{remark}
\newtheorem{remark}[theorem]{Remark}

\numberwithin{equation}{section}
\numberwithin{theorem}{section}

\newcommand{\R}{\mathbb{R}}
\newcommand{\Rd}{\mathbb{R}^{d}}

\newcommand{\conv}{\operatorname{conv}}
\newcommand{\supp}{\operatorname{supp}}

\newcommand{\aff}{\textnormal{aff}}

\newcommand{\E}{\mathbb{E}}
\newcommand{\MT}{\mathsf{MT}}
\newcommand{\Cpl}{\mathsf{Cpl}}

\newcommand{\D}{\mathcal{D}}

\newcommand{\bary}{\operatorname{bary}}
\newcommand{\lc}{\preceq_{\textnormal{c}}}
\newcommand{\MCov}{\textnormal{MCov}}

\newcommand{\Law}{\operatorname{Law}}
\newcommand{\dom}{\operatorname{dom}}
\renewcommand{\P}{\mathcal{P}}
\newcommand{\PP}{\mathcal{P}}

\DeclareMathOperator{\Argmin}{Argmin}
\DeclareMathOperator{\diag}{diag}
\DeclareMathOperator{\Cov}{Cov}

\begin{document}

	\title{Bridging classical and martingale Schrödinger bridges}
	
	\author{Julio Backhoff, Mathias Beiglböck, Giorgia Bifronte, Armand Ley.}
	
	\begin{abstract} 
		
		We investigate the martingale Schrödinger bridge, recently introduced by Nutz and Wiesel as a distinguished martingale transport plan between two probability measures in convex order. We show that this construction extends naturally to arbitrary dimension and admits several equivalent characterizations. In particular, we identify its continuous-time counterpart as the continuous martingale with prescribed marginals that minimizes a weighted quadratic energy measuring the deviation  from  Brownian motion. In the irreducible case, we prove that this continuous martingale Schrödinger bridge coincides with the Föllmer martingale, that is, with the Doob martingale associated to a suitable Föllmer process. More generally, we relate the martingale Schrödinger bridge to a variational problem over base measures and to the dual formulation of the corresponding weak optimal transport problem, thereby clarifying its connection with the classical Schrödinger bridge.
		
		\bigskip
		
		
		\smallskip
		
	\end{abstract}
	
	\maketitle

	\section{Introduction}
	
	We consider martingale transport between probability measures \(\mu,\nu\in\PP_2(\R^d)\).
	By Strassen's theorem \cite{St65}, the set \(\MT(\mu,\nu)\) of martingale couplings is nonempty if and only if \(\mu\lc\nu\).
	Strassen's proof is nonconstructive, and numerous works have sought concrete or canonical examples, see \cite{BeJu16, BeCoHu14, GhKiLi19, BaBeHuKa20, HuTr17, HeTo13, He19} among many others.
	Besides its intrinsic mathematical interest, this is also highly relevant for applications, for instance in mathematical finance (see e.g.\ \cite{CoHe21, GuLoWa19, AcMaPa25, HaJoLoObPa23}) and, more recently, in generative modelling \cite{AlHLLoMaPhTo26b, AlHLLoMaPhTo26a, DMPhZa26}.
	Interestingly, most known constructions are essentially limited to dimension \(d=1\) and in  either discrete time or continuous time, but rarely in a form that treats both settings simultaneously.
	
	Our main object is the \emph{martingale Schr\"odinger bridge} \(\msb\), introduced by Nutz and Wiesel \cite{NuWi24} as the optimizer of a discrete-time entropic martingale transport problem.  
    Informally, \(\msb\) is the most independent coupling of \(\mu\) and \(\nu\) subject to the martingale constraint.
	We extend this construction to arbitrary dimension and show that it admits several equivalent characterizations in both discrete and continuous time. 
	
	\subsection{Overview of five main characterizations of \(\msb\)}

	For this overview, we restrict ourselves to the  case where the entropic martingale optimal transport (MOT) problem \eqref{eq:introP} is finite and $(\mu, \nu)$ is irreducible. 
	In this regime, the five characterizations below are equivalent whenever the relevant dual maximizers exist. \eqref{eq:introP} belongs to the broader class of weak martingale transport problems, where dual attainment is a delicate problem, see \cite{BeNuTo16, BePaRiSc25, CaMaSy25}. Remarkably, Hasenbichler, Pammer, and Schrott  (\cite{HaPaSc26}, forthcoming) establish the existence of dual maximizers in $\R^d$ in the irreducible case  under mild regularity assumptions on the cost function which are satisfied by the entropy. 
	To keep the presentation concise, we suppress the precise integrability and regularity assumptions and focus only on the structural picture.

	
	\subsubsection{Entropic MOT characterization}
	
	We begin with the discrete-time martingale Schr\"odinger problem
	\begin{align}\label{eq:introP} \tag{\text{P}}
		\inf_{m\in \MT(\mu,\nu)} H(m\mid \mu\otimes \nu),
	\end{align}
	where $H(Q\mid R)$ denotes the relative entropy, i.e.\ equals $\int \log\frac{dQ}{dR}\, dQ$ if $Q\ll R$ and $+\infty$ otherwise.
	We call the unique minimizer \(\msb\) of this problem the martingale Schr\"odinger bridge.
	We note that, if $H(\nu|\gamma) < +\infty$, then taking the reference measure to be \(\mu.\gamma\), having first marginal $\mu$ and independent Gaussian increments, changes the objective only by a constant and can hence be used to define \(\msb\), see Lemma \ref{lem:relation_gaussian_analogue} below.
	
	
	\subsubsection{Gibbs form of the density through a dual WOT problem}

	Problem \eqref{eq:introP} admits a dual formulation, as already noted by Nutz and Wiesel \cite{NuWi22} in the one-dimensional case. We rather view Problem \eqref{eq:introP} as a weak transport problem in the sense of \cite{GoRoSaSh18}; see \eqref{eq:main_dual} below. The dual theory of such problems has received significant attention 
	\cite{GoRoSaTe17, AlBoCh18, BaBePa18, BaPa20, BePaRiSc25} in particular the existence of optimizers is treated by Carlier, Malamut and Sylvestre in \cite{CaMaSy25} as well as the aforementioned \cite{HaPaSc26}. Dual attainment (either in the Nutz-Wiesel form or in the weak dual \eqref{eq:main_dual}) 
	implies that \(m\in \MT(\mu,\nu)\) is the martingale Schr\"odinger bridge if and only if it has a density of the form
	\begin{equation}\label{eq:introGibbs}
		\frac{dm}{d\mu\otimes \nu}(x,y)
		=
		\exp\big(\varphi(x)+\psi(y)+h(x)\cdot(y-x)\big).
	\end{equation}
	This is the martingale analogue of the usual exponential structure in the classical Schr\"odinger problem, with an additional term \(h(x)\cdot(y-x)\) playing the role of a Lagrange multiplier for the martingale constraint.
	
	\subsubsection{Continuous-time minimizer}
	
	A continuous-time martingale Schr\"odinger bridge is obtained by minimizing a weighted quadratic energy among continuous martingales with prescribed marginals.
	Consider
	\begin{align}\label{eq:introP_cont}\tag{$P^{cont}$}
		\inf_{\substack{M_0\sim\mu,\;M_1\sim\nu,\, M_t=M_0+\int_0^t \sigma_s\,dB_s}}
		\frac12\,\E\Big[\int_0^1 \frac{|\sigma_t-I_d|_{\text{HS}}^2}{1-t}\,dt\Big].
	\end{align}
	If \(M\) is the (unique in law) optimizer of \eqref{eq:introP_cont}, then \(\Law(M_0,M_1)\) is precisely the optimizer \(\msb\) of \eqref{eq:introP}. 
	The converse direction is more interesting.
	Starting from the static optimizer \(\msb(dx,dy)=\mu(dx)\msb_x(dy)\), one may reconstruct the continuous-time optimizer fiberwise.
	Specifically, for \(x\in\supp(\mu)\), let \((M_t^x)_{t\in[0,1]}\) be the F\"ollmer martingale, defined below, with deterministic start and terminal law \(M_1^x\sim \msb_x\).
	Mixing the family \((M^x)_x\) against \(x\sim\mu\) then yields the optimizer of \eqref{eq:introP_cont}.
	
	\subsubsection{F\"ollmer martingale characterization}
	
	Given probability measures \(\bar\mu,\nu\in\PP_2(\R^d)\), consider the classical dynamic Schr\"odinger problem
	\begin{align}\label{eq:ContSchr}
		\inf_{\substack{X_0\sim \bar\mu,\;X_1\sim \nu, X_t=X_0+\int_0^t \beta_s\,ds+B_t}}
		\frac12\,\E\Big[\int_0^1 |\beta_t|^2\,dt\Big].
	\end{align}
	Its optimizer \((\Xsb_t)_{t\in[0,1]}\) is called the Schrödinger bridge, or Föllmer process, between \(\bar\mu\) and \(\nu\).
	We will use the name \emph{Föllmer martingale} for the associated Doob martingale
	\[
	M_t:=\E[\Xsb_1\mid \mathcal F_t].
	\] 
	Remarkably, every F\"ollmer martingale is the continuous-time martingale Schr\"odinger bridge between its own marginals.
	Conversely, if \((\mu,\nu)\) is irreducible, there exists a unique Föllmer martingale satisfying \(M_0\sim \mu\) and \(M_1\sim \nu\).
	The Föllmer martingale is a (time-inhomogenous) Markov martingale that can be seen as a deterministic transformation of a Schrödinger bridge.
	Specifically, we have
	\[
	M_t=f_t(\Xsb_t),
	\,\,
	f_t(z):=\E[\Xsb_1\mid \Xsb_t=z]
	= z+(1-t)\,\nabla_z\log P_{1-t}(d\nu/d\gamma\cdot e^\psi) ,
	\]
	where \( \psi \) is the second Schr\"odinger potentials of the Schrödinger bridge between $\bar \mu$ and $\nu$, and \(P_t\) denotes the heat semigroup.
	
	The link between the continuous-time and discrete-time formulations is rooted in the representation of relative entropy via adapted drifts, going back to the work of Föllmer \cite{Fo85,Fo86}. See also the survey by L{\'e}onard \cite{Le14}. In continuous time, entropy minimization admits a stochastic control representation in terms of the Föllmer drift, whose quadratic energy coincides with relative entropy, as made explicit by Lehec \cite{Le13}. This identity, further developed by Eldan, Mikulincer, and Shenfeld \cite{ElMi20,ElLeSh20}, provides a dynamic formulation of entropy that naturally bridges to the discrete-time entropic transport problems.

	\subsubsection{F\"ollmer coupling and the base measure via \(\MCov\)-conjugation}
	
	The Föllmer martingale construction can also be formulated naturally at the discrete-time level.
	To this end, define for \(\mu,\eta,\nu\in\PP_2(\R^d)\) the max-covariance functional and its entropic counterpart
	\begin{align}\label{eq:McovsIntro}
		\MCov(\mu,\eta)=\sup_{m\in\Cpl(\mu,\eta)} \int \bar x\cdot y\,m(d\bar x, dy), \quad  \mathcal E_\nu(\eta)
		=
		\sup_{\pi\in\Cpl(\eta,\nu)}
		\Big\{\int \bar x\cdot y\,d\pi-H(\pi\mid \eta\otimes \nu)\Big\}.
	\end{align}
	While the optimizer of \(\MCov\) is precisely the classical quadratic-cost optimizer, that is, the Brenier map, the optimizer of \(\mathcal E_\nu\) is the classical Schr\"odinger coupling \(\psb\) between \(\eta\) and \(\nu\), equivalently the endpoint law of the dynamic optimizer in \eqref{eq:ContSchr}.
	In particular,
	\begin{align*}
		m_{\psb} := ((\bar x,y)\mapsto (T(\bar x), y))_{\# }\psb, \quad  T(\bar x)=\bary(\psb_{\bar x})
	\end{align*}
	is the static martingale Schrödinger bridge between its marginals.
	We call \(\bar\mu\) the base measure for the martingale Schrödinger bridge between \(\mu\) and \(\nu\) if \(m_{\psb} \in \MT(\mu,\nu)\) for $\psb$ optimal for $\mathcal E_{\nu}(\bar\mu)$, and hence \(m_{\psb}=\msb\).
	Strikingly, the base measure \(\bar\mu\) can be characterized  through the \(\MCov\)-conjugate of \(\mathcal E_\nu\):
	\begin{align}\label{eq:vp}\tag{VP}
		\sup_{\eta\in\PP_2(\R^d)} \big\{\MCov(\mu,\eta)-\mathcal E_\nu(\eta)\big\}.
	\end{align}
	Any optimizer \(\eta=\bar\mu\) yields the base measure whose F\"ollmer coupling with \(\nu\) induces the martingale Schr\"odinger bridge from \(\mu\) to \(\nu\).
	Thus \(\bar\mu\) appears as the optimal intermediate law through which the martingale Schr\"odinger problem factors into a martingale part and a classical Schr\"odinger part.
	
	This variational structure becomes more transparent after lifting to the Lions space, see \cite{BaPaSc24, PiSa25, BePaSc25}.
	Indeed,  the functionals \(\eta\mapsto \MCov(\mu,\eta)\) and \(\eta\mapsto \mathcal E_\nu(\eta)\) admit liftings to functionals on \(L^2\), and in that lifted space the above \(\MCov\)-conjugation becomes an actual convex conjugation.
	This viewpoint also motivates the first-order conditions appearing in Section \ref{sec:implications} below.
	At an optimizer \(\bar\mu\) of \eqref{eq:vp}, one expects (see \cite[Section 2]{BePaSc25}) that the Wasserstein gradients satisfy
	\[ \textstyle 
	\frac{\delta}{\delta\bar\mu}\MCov(\mu,\bar\mu)
	=
	\frac{\delta}{\delta\bar\mu}\mathcal E_\nu(\bar\mu)
	\]
	up to an additive constant.
	Equivalently, for any lift \(\bar X\sim\bar\mu\), the corresponding lifted gradients coincide at \(\bar X\).
	The left-hand side comes from the martingale transport part, while the right-hand side is encoded by the Schr\"odinger potentials of the classical bridge from \(\bar\mu\) to \(\nu\).
	This first-order condition is the key mechanism by which the optimal base measure \(\bar\mu\) is linked to the Gibbs structure of \(\msb\).
	
	\subsection{How the five characterizations fit together}
	
	At a heuristic level, the whole structure is determined once the base measure \(\bar\mu\) is known.
	Indeed, let \(\bar \varphi,\psi\) be the Schr\"odinger potentials of the classical Schr\"odinger bridge \(\psb\) between \(\bar\mu\) and \(\nu\).
	Then \(\psb\) is given in Gibbs form by
	\[
	\psb(d \bar x,dy)=\bar\mu(d \bar x)\psb_{\bar x}(dy),
	\qquad
	\frac{d\psb_{\bar x}}{d\nu}(y)=\frac{e^{\bar x\cdot y+\psi (y)}}{\int e^{ \bar x\cdot z+\psi(z)}\,\nu(dz)}=e^{\bar x\cdot y+ \psi (y)+\bar \varphi (\bar x)}.
	\]
	Hence its barycentric projection is explicit:
	\[
	T(\bar x):=\bary(\psb_{\bar x})=\int y\,\psb_{\bar x}(dy)
	=\nabla_{\bar x}\log\int e^{ \bar x \cdot z+\psi(z)}\,\nu(dz).
	\]
	
	The role of the base measure is precisely that \(T_{\#}\bar\mu=\mu\).
	Once this holds, the martingale Schr\"odinger bridge is obtained simply by barycentric projection of the classical Schr\"odinger coupling,
	\[
	\msb=(T,\id)_{\#}\psb.
	\]
	In particular, the static Gibbs form of \(\msb\) is just the image of the Gibbs form of \(\psb\). As we prove below \(T\) is invertible, which yields exactly a density of the form \eqref{eq:introGibbs}.

	The resulting dictionary can be summarized by the following picture:

		
		
		
		

\[
\makebox[\textwidth][c]{%
\resizebox{\textwidth}{!}{%
\begin{tikzpicture}[baseline=(current bounding box.center)]
    \node (mubar) at (0,2) {$\bar\mu$};
    
    \node (pisym) at (4.2,2) {$\psb$};
    \node (piset) [anchor=west] at (4.45,2) {$\in\Cpl(\bar\mu,\nu)$};
    
    \node (mu) at (0,0) {$\mu$};
    
    \node (msym) at (4.2,0) {$\msb$};
    \node (mset) [anchor=west] at (4.45,0) {$\in\MT(\mu,\nu)$};
    
    \draw[->] (mubar) -- node[above] {$\text{solve }SP(\bar\mu,\nu)$} (pisym);
    \draw[->] (mubar) to[bend right=18] node[left] {$T(\bar x)=\bary(\psb_{\bar x})$} (mu);
    \draw[->] (pisym.south) -- node[right] {$(T,\id)_\#$} (msym.north);
    \draw[->] (mu) to[bend right=18] node[right] {$h=T^{-1}$} (mubar);
    \draw[->] (mu) -- node[below] {$x\mapsto \msb_x=\psb_{h(x)}$} (msym);
\end{tikzpicture}
\hspace{0.35cm}
\begin{tikzpicture}[baseline=(current bounding box.center)]
    \node[font=\large] (mubar) at (0,2) {$\bar\mu$};
    \node[font=\large] (mu)    at (0,0) {$\mu$};
    \node[font=\large] (nu)    at (4.8,1) {$\nu$};

    \node[font=\normalsize] at (2.4,2.25)
    {$dX_t = dB_t + u_t(X_t)\,dt$};

    \node[font=\normalsize] at (2.4,-0.25)
    {$M_t=\mathbb{E}[X_1\mid\mathcal{F}_t]$};

    \draw[->] (mubar) to[bend left=7] (nu);

    \draw[->] (mubar) to[bend right=18]
    node[left,font=\normalsize] {} (mu);

    \draw[->] (mu) to[bend right=18] (mubar);
    
    \draw[->] (mu) to[bend right=7] (nu);
\end{tikzpicture}
}%
}
\]

	The continuous-time objects are then equally explicit.
	The same potential \(\psi\) determines the F\"ollmer process \(\Xsb\) from \(\bar\mu\) to \(\nu\), and the associated F\"ollmer martingale is its Doob transform
	\[
	M_t=\E[\Xsb_1\mid\mathcal F_t]=f_t(\Xsb_t),
	\,\,
	f_t(z):=z+(1-t)\, \nabla_z \log P_{1-t}(d\nu/d\gamma\cdot e^\psi ) .
	\]
	Thus, once \(\bar\mu\) and \(\psi\) are identified, the F\"ollmer coupling, the martingale Schr\"odinger bridge, its Gibbs density, and its continuous-time realization all follow in a rather explicit way.
	
	\begin{figure}[h]
		\centering
		\includegraphics[width=0.8\textwidth]{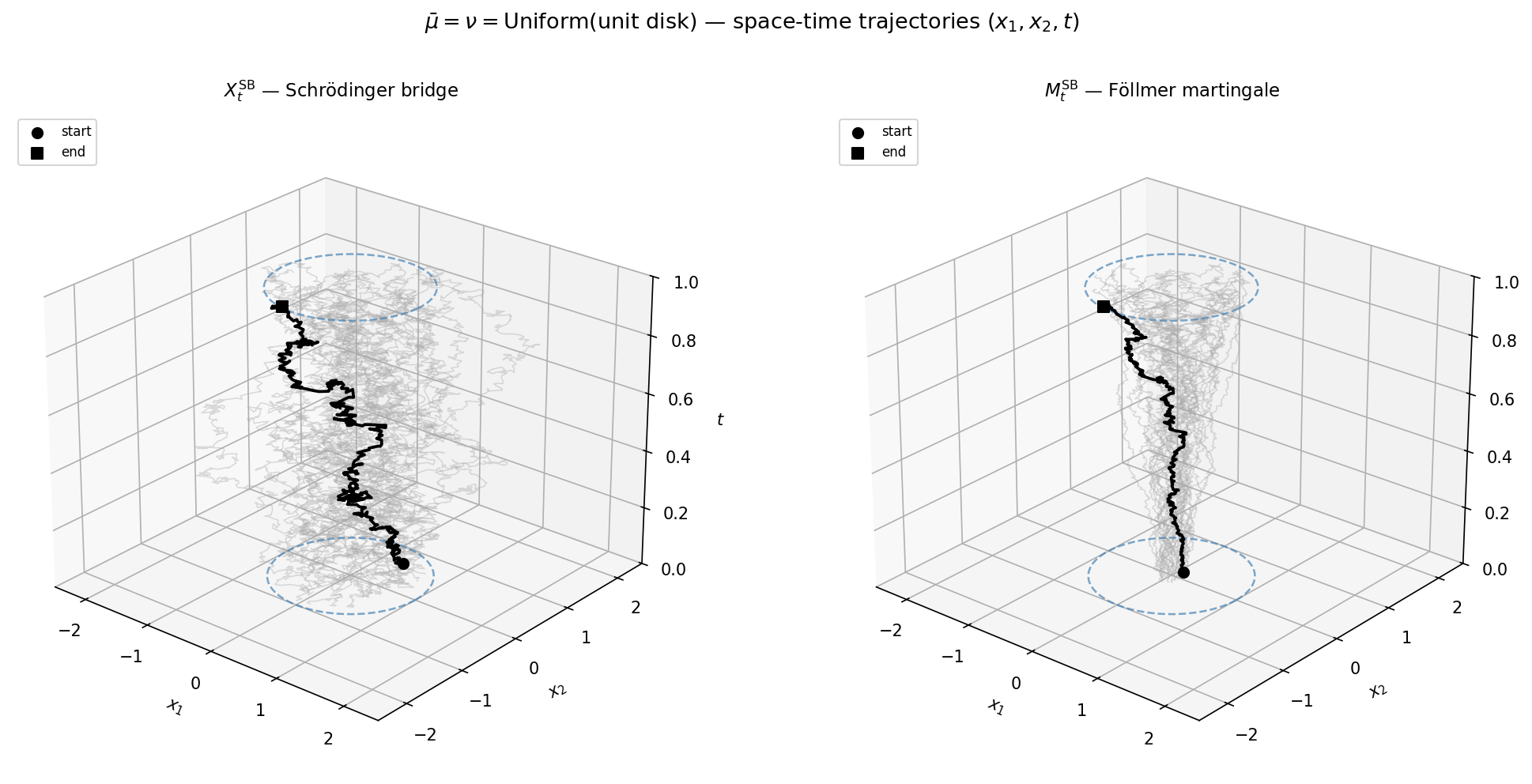}
	\end{figure}

	\subsection{Behavior under scaling and connection to filtering}\label{sec:intro_filter}
	
	A key structural feature of the F\"ollmer martingale is its simple dependence on the volatility of the reference Brownian motion.
	Write \(\msb(dx,dy)=\mu(dx)\msb_x(dy)\) for the disintegration of the (static) martingale Schrödinger bridge and let \(\sigma>0\).
	For each \(x\), consider the F\"ollmer process from \(\delta_x\) to \(\msb_x\) relative to the reference process \(\sigma B\), and denote by \(M^\sigma=(M_t^\sigma)_{t\in[0,1]}\) the corresponding Doob martingale.
	Thus \(M^\sigma\) is the F\"ollmer martingale associated with  \(\msb\) and the reference volatility \(\sigma\).
	The dependence on $\sigma$ becomes particularly  transparent after passing to a filtering representation which is of independent interest:


	{Let \((X,Y)\sim \msb\), let \(W\) be an independent Brownian motion in \(\R^d\). The variable $Y - X$ can be interpreted as a hidden static signal observed through the noisy channel
		\[
		R_s = s(Y - X) + W_s, \qquad s \geq 0.
		\]
		The minimum mean-square error estimator of this signal is the filter
		\[
		\tilde{Z}_s := \E\bigl[Y - X \,\big|\, X ,\, (R_r)_{r \leq s}\bigr].
		\]
		When $\nu$ is finitely supported, this estimation problem reduces to the classical Wonham filter \cite{Wo64}. The estimation of time-constant signal corrupted with Brownian noise have recently attracted renewed attention, as the law of the random tilts driving Eldan's stochastic localization scheme coincides with the law of the observation process when the hidden signal is distributed according to the localized measure. We refer to Klartag and Putterman \cite[Section 4]{KlPu2023}, who develop this Bayesian viewpoint in detail\footnote{We thank Daniel Lacker for pointing out to us this connection to stochastic localization.}. 	
		We will show in \Cref{thm:intro_scaling_filter} that, via the deterministic time change $s \in [0, +\infty ] \mapsto \frac{\sigma^2 s}{1+\sigma^2 s} \in [0,1]$, the process $$Z_s := X + \tilde{Z}_s = \E[Y \mid X, (R_r)_{r \leq s}]$$ is precisely the canonical form of the F\"ollmer martingale, independently of the choice of \(\sigma\). Thus all F\"ollmer martingales associated with the same data \((\mu, \nu)\), but different reference volatilities, are in fact one and the same process merely expressed in different time scales.}

	\subsection{Relation to Henry-Labord\`ere's Schr\"odinger martingale}

	It is instructive to compare our construction with the \emph{Schr\"odinger martingale} introduced by Henry-Labord\`ere in \cite{He19}.
	The two approaches are closely related in spirit, but they solve different optimization problems and should not be generally identified.
	
	Henry-Labord\`ere's starting point is a fixed prior martingale model on path space.
	Given such a prior law \(P^0\), one considers calibrated martingale laws \(P\) and minimizes the path-space relative entropy
	\[
	\inf_{P\in\mathcal M_{\rm cal}} H(P\mid P^0),
	\qquad
	H(P\mid P^0):=E^P\!\left[\log\frac{dP}{dP^0}\right].
	\]
	Thus the prior path law is part of the input, and the resulting Schr\"odinger martingale depends essentially on this choice of reference model. The change of measure w.r.t.\ the reference model translates into a change of drift for the volatility process. This approach has proved fruitful, e.g.\ in the applications by Guyon \cite{GuMeNu17}.
	
	Our construction is different in a fundamental way.
	The primary object in the present paper is the static martingale coupling \(\msb\), selected intrinsically by the endpoint optimization problem \eqref{eq:introP}.
	Only after this static step do we pass to continuous time and associate with \(\msb\) its canonical F\"ollmer martingale.
	Accordingly, our continuous-time martingale is canonical once \((\mu,\nu)\) are fixed, whereas Henry-Labord\`ere's Schr\"odinger martingale is prior-dependent.

	The two frameworks overlap only at the philosophical  level of the underlying Schr\"odinger  mechanism.
	This becomes most transparent in the Dirac-start case.
	If one fixes \(x\in\R^d\) and a terminal law \(\rho\), then  both approaches naturally involve a deformation of Brownian motion started from \(\delta_x\) and constrained by the terminal behavior. However Henry-Labord\`ere minimizes an entropy functional while our approach minimizes a weighted energy. We refer to the articles \cite{AlHLLoMaPhTo26a,AlHLLoMaPhTo26b,DMPhZa26} for recent alternative approaches and Carmona and Xu \cite{CaXu97} for a precursor.

{
    \subsection{Practical aspects}

    There are multiple ways how one would algorithmically obtain the martingale Schrödinger Bridge. One possibility is to do alternate updating of the three Lagrange multipliers appearing in \eqref{eq:introGibbs},  in the spirit of the Sinkhorn iterations for the classical Schrödinger problem. This has been proposed and studied, in dimension one and under suitable assumptions, by Chen, Conforti, Ren and Wang in \cite{ChCoReWa26}. Carlier, Malamut and Sylvestre \cite{CaMaSy25} propose to use the so-called SISTA algorithm (see \cite{CaGuGaSu23,GaNaTo25}), wherein the update of the Lagrange multiplier for the martingale constraint is done approximately through a gradient step. Yet another possibility is to work with our variational problem \eqref{eq:vp}: one can e.g.\ consider the immediate analogue of the Conze-Henry-Labord{\`e}re iterations for Bass martingales \cite{CoHe21}, or the analogue to the gradient flow of the Bass functional \cite{BaPaSc24}. By analogy with (entropic) optimal transport, we expect that the algorithms for martingale Schrödinger Bridges should run faster than the algorithms for Bass martingales in higher dimensions.

    Although Nutz and Wiesel \cite{NuWi24} envisioned martingale Schrödinger Bridges with financial calibration as their main motivation (a point of view now strengthened thanks to the continuous time variant \eqref{eq:ContSchr}), an application beyond finance already exists: 
In \cite{GuJuTa25} Guo, Juillet and Tang study \eqref{eq:introP} when $\mu$ and $\nu$ are uniform measures with finite support on $\R$, with the aim of providing a canonical construction of the football model of round-Robins tournaments. In this context they prove that the martingale Schrödinger bridge $\msb$ induces a score matrix satisfying the strong stochastic transitivity property and provide an algorithm to compute it, thereby answering a question of Aldous and Kolesnik \cite{AlKo22}. }

	\subsection{Outline of the paper}
	
	In Section \ref{sec:eq-discrete}, we show that Problem \eqref{eq:introP}, its dual counterpart \eqref{eq:main_dual}, and the variational problem \eqref{eq:vp} have the same values. 
	In Section \ref{sec:implications}, we start by showing that attainment of \eqref{eq:vp} implies attainment of \eqref{eq:main_dual}, and how the solution of \eqref{eq:main_dual} can be expressed through the solution of \eqref{eq:vp}. We then prove the converse statement by proving that attainment of \eqref{eq:main_dual} yields attainment of \eqref{eq:vp}, and show how the solution of \eqref{eq:vp} can be expressed through the solution of \eqref{eq:main_dual}.
	In Section \ref{sec:continuous} we analyze the continuous-time problem \eqref{eq:introP_cont}, establishing its equality with \eqref{eq:introP}, and developing in detail the concept of Föllmer martingale.
	In Section \ref{sec:filtering} we elaborate in more detail on the prior-free nature of the Föllmer martingale and its connection to filtering.
	In Section \ref{sec:examples} we study two examples; the Gaussian case and a simple discrete case, where we can explicitly describe the optimizer as well as compare it to the better known \emph{Bass martingales}.\\

	\emph{Notation:} Throughout $\mu,\nu\in \PP_{2}(\Rd) $ are assumed to be in convex order, denoted by $\mu \lc \nu$, and meaning that $\int \phi \, d\mu \leqslant \int \phi \, d\nu$ holds for all convex functions $\phi \colon \Rd \rightarrow \R$. We denote by $H(\cdot|\cdot)$ the relative entropy between probability measures (also called KL divergence). $\gamma_x^t$ stands for the Gaussian distribution centred at $x$ and with covariance matrix $tI$, but we write for brevity $\gamma:=\gamma_0^1$, $\gamma^t:=\gamma_0^t$, and $\gamma_x:=\gamma_x^1$. We denote by $\mathcal W_2$ the quadratic Wasserstein distance. We write $\bar\mu.\gamma(dx,dy):=\bar\mu(dx)\gamma_x(dy)$, and if $\pi$ is a probability measure on $\Rd\times\Rd$ we write $\pi_x$ for its disintegration with respect to the first (i.e.\, its initial) marginal, while $p_1(\pi)$ and $p_2(\pi)$ stand for the first and second marginals of $\pi$. $\Cpl(\mu,\nu)$ stands for the set of coupling with first resp.\ second marginal $\mu$ resp.\ $\nu$, while $\MT(\mu,\nu)$ stands for its subset of martingale couplings. Finally, if $g$ is a function and $\rho$ is a measure, we write $g\ast \rho(x)=\int g(x-z)\rho(dz)$ for their convolution.\\

	\section{Preliminaries and equality of values for the discrete-time problems}
	\label{sec:eq-discrete}
	
	This work is concerned with the relation between multiple optimization problems. The main one is
	\begin{align}\tag{P}
		\inf_{m\in\MT(\mu,\nu)} H(m|\mu\otimes \nu),
	\end{align}
	which is clearly equivalent to $\inf_{m\in\MT(\mu,\nu)} \int H(m_x|\nu)\mu(dx)$. 
	As we will see, this problem has a related continuous-time counterpart
	\begin{align}\tag{$P^{cont}$}
		\inf_{\substack{M_{0} \sim \mu, \, M_{1} \sim \nu, \, M_{t} = M_{0} + \int_{0}^{t}  \sigma_{s} \, dB_{s}}} \frac{1}{2}
		\mathbb{E}\Big[\int_{0}^{1} \frac{\vert \sigma_{t} - I \vert^{2}}{1-t} \, dt\Big],
	\end{align}
	where for matrix norm we use the Hilbert-Schmidt one,  the infimum runs over the class of filtered probability spaces supporting a $d$-dimensional Brownian motion $B$, and where $M$ is a continuous martingale satisfying the stated conditions.
	On the dual side, we will encounter the problem
	\begin{align}\label{eq:main_dual}\tag{D}
		\sup_{\psi\in C_b(\Rd)}\left\{\int\psi d\nu +\int  \sup_{h\in\Rd}\left[ h\cdot x-\log\int\exp(h\cdot y+\psi(y))\nu(dy) \right ] \mu(dx)\right\},
	\end{align}
	Finally, we will be concerned with the variational problem
	\begin{align}\tag{VP}
		\sup_{\bar\mu\in\PP_2(\Rd)}\left\{ \inf_{\pi\in \Cpl(\bar\mu,\nu) }\left[ H(\pi|\bar\mu \otimes \nu)-\int \bar x\cdot y \,\pi(d\bar x,dy) \right] + \MCov(\bar\mu,\mu)  \right\},
	\end{align}
	where here and throughout $$\MCov(\bar\mu,\mu):=\sup_{\pi\in\Cpl(\bar\mu,\mu)}\int \bar x\cdot x\,\pi(d\bar x,x).$$
	We recall that $\pi$ (uniquely) attains the Schrödinger problem 
	\begin{align}
		\inf_{\pi\in \Cpl(\bar\mu,\nu) }\left[H(\pi|\bar\mu\otimes \nu) - \int \bar x\cdot y \,\pi(d\bar x,dy)\right] \label{eq:SP}\tag{$SP(\bar \mu,\nu)$}
	\end{align} iff $\frac{d\pi}{d\bar\mu \otimes \nu}(\bar x,y)= \exp( \bar\varphi(\bar x) + \psi(y) +\bar x\cdot y)$, where $(\bar\varphi,\psi)$ solve the associated Schrödinger system
	\begin{align}\label{eq:SchS1}\tag{SS1}
		\int \exp( \bar\varphi(\bar x) + \psi(y) +\bar x\cdot y) \nu(dy)&=1 \,\,(\bar\mu(d\bar x)-a.s.);\\
		\int \exp(\bar\varphi(\bar x) + \psi(y) + \bar x\cdot y) \bar\mu(d\bar x)&=1 \,\,(\nu(d y)-a.s.) \label{eq:SchS2}\tag{SS2}
	\end{align}
	The solution $(\bar\varphi,\psi)$ is unique except for the fact that $(\bar\varphi + c,\psi -c)$ is also a solution for all $c \in \R$. 
	Since the integral in \eqref{eq:SchS1} is everywhere defined, in the following we can and will uniquely extend $\bar\varphi(\cdot)$ by requiring that \eqref{eq:SchS1} holds for all $\bar x\in\Rd$. Similarly so for $\psi$. H\"older's inequality shows that $\bar x\mapsto\log \int \exp(\psi(y) + \bar x\cdot y) \nu(dy)$ is convex, hence $\bar\varphi$ is concave (and so is $\psi$).

	In this part we first establish that the primal and dual problems share the same value, and then prove that this common value coincides with that of the variational problem.
	
	\begin{lemma}\label{lem:V_D_eq}
		value\eqref{eq:introP}$=$ value\eqref{eq:main_dual}.
	\end{lemma}
	
	\begin{proof}
		Duality for weak optimal transport (see, e.g., \cite{BaBePa18}) gives that \eqref{eq:introP} is equal to the dual problem
		$$\sup_\psi\left\{\int\psi d\nu + \int \inf_{p\, s.t.\, \text{mean}(p)=x}\left [ H(p|\nu) -\int\psi dp \right]\mu(dx)\right\}.$$
		Now for the minimization problem under the integral, by duality  $$\inf_{p\, s.t.\, \text{mean}(p)=x}\left [ H(p|\nu) -\int\psi dp \right]=\sup_{h\in\Rd}\left[h\cdot x-\log\int \exp(h\cdot y +\psi(y))\nu(dy) \right].\qedhere$$
	\end{proof}
	
	In fact, duality in the above proof also suggests the existence of optimal $h(x)$ there, and that $p_x$ is optimal iff $\frac{dp_x}{d\nu}(y)=\frac{\exp(h(x)\cdot y + \psi(y))}{Z_{x,\psi}}$, with $Z_{x,\psi}$ a normalizing constant. In particular if $\psi$ is optimal, then the map $h$ pushes forward $\mu$ to $\bar\mu$. We will formalize these intuitions soon. Now we show that the primal and dual problems share the same value as the variational problem.
	
	\begin{lemma}\label{lem:P_VP_eq}
		value\eqref{eq:introP}$= $ value\eqref{eq:vp}.
	\end{lemma}

	\begin{proof}
		We will prove this equality using the dual formulation of \eqref{eq:introP}. Before that, we provide a direct argument for the inequality value\eqref{eq:introP}$\geq $ value\eqref{eq:vp}, which we find illustrative but can be skipped: 
		Let $m\in \MT(\mu,\nu)$ and $\kappa$ a $\MCov$-optimal coupling from $\bar \mu$ to $\mu$. Then introduce $\pi(d\bar x,dy):=\bar\mu(d\bar x)\int_x m_x(dy)\kappa_{\bar x}(dx)\in\Cpl(\bar\mu,\nu)$ and $\int y\pi_{\bar x}(dy)=\int x \kappa_{\bar x}(dx)$, so $\int \bar x\cdot y \,\pi(d\bar x,dy)=\int \bar x\cdot x\, \kappa(d\bar x,dx)=\MCov(\bar\mu,\mu)$. By convexity of the entropy we have
		$$ \int H(m_x|\nu)\mu(dx) = \iint H(m_x|\nu)\kappa_{\bar x}(dx)\bar\mu(d\bar x)   \geq \int H( \pi_{\bar x}|\nu)\bar\mu(d\bar x).$$ All in all
		\begin{align} \notag
			\int H(m_x|\nu)\mu(dx) & \geq \int H( \pi_{\bar x}|\nu)\bar\mu(d\bar x)-\iint \bar x\cdot y\, \bar\pi(d\bar x,dy)+\MCov(\bar\mu,\mu) \\ & \geq \inf_{\pi\in \Cpl(\bar\mu,\nu) }\left[H(\pi|\bar\mu\otimes \nu) - \iint \bar x\cdot y \,\pi(d\bar x,dy)\right]   + \MCov(\bar\mu,\mu). \label{eq:ineq_VP}
		\end{align}
		We conclude by taking infimum over the relevant $m$ and supremum over $\bar\mu$. For later reference, observe that the first inequality in \eqref{eq:ineq_VP} becomes an equality if $\kappa$ is induced by a map.
		
		Now we prove the desired equality. For that matter we show that value\eqref{eq:vp} $=$ value\eqref{eq:main_dual}, and conclude thanks to Lemma \ref{lem:V_D_eq}. We will repeatedly use the fact that $\inf_{r\in\R}\{ e^{r}p -r-1\}=\log(p)$,
		for $p\geq 0$. Duality for \eqref{eq:SP} says
		\begin{align*}
			value \eqref{eq:SP} & = \sup_{\bar\varphi,\psi}\left \{ \int \bar\varphi d\bar\mu +\int \psi d\nu-\log\int e^{\bar\varphi(\bar x)}\left[  e^{\psi(y)+\bar x\cdot y}\right ]\bar\mu(d\bar x)  \right \} \\
			&= \int \psi d\nu + \int \bar\varphi d\bar\mu + \sup_{r\in\R}\left \{1+r-e^r \int e^{\bar\varphi(\bar x)}\int  e^{\psi(y)+\bar x\cdot y}\nu(dy)\bar\mu(d\bar x)  \right \}.
		\end{align*}
		Thus the value of \eqref{eq:vp} is equal to
		\begin{align*}
			&  \sup_{\bar\varphi,\psi,r}\left\{\int\psi d\nu+1+ \sup_{\bar\mu}\left (r+ \int \bar\varphi d\bar\mu -e^r \int e^{\bar\varphi(\bar x)}\int  e^{\psi(y)+\bar x\cdot y}\nu(dy)\bar\mu(d\bar x) +\MCov(\bar\mu,\mu)\right)  \right\} \\
			=& \sup_{\bar\varphi,\psi}\left\{\int\psi d\nu+1+ \sup_{\substack{X\sim\mu ,\bar X \mbox{ free}}} \E\left (  \bar\varphi(\bar X) -  e^{\bar\varphi(\bar X)}\int  e^{\psi(y)+\bar X\cdot y}\nu(dy) +\bar X\cdot X\right)  \right\}\\
			=& \sup_{\bar\varphi,\psi}\left\{\int\psi d\nu+1+  \int\sup_{\bar x}\left (  \bar\varphi(\bar x) -  e^{\bar\varphi(\bar x)}\int  e^{\psi(y)+\bar x\cdot y}\nu(dy) +\bar x\cdot x\right)  \mu(dx)\right\} \\
			=& \sup_{\psi}\left\{\int\psi d\nu+1+  \sup_{\bar\varphi}\int\sup_{\bar x}\left (  \bar\varphi(\bar x) -  e^{\bar\varphi(\bar x)}\int  e^{\psi(y)+\bar x\cdot y}\nu(dy) +\bar x\cdot x\right)  \mu(dx)\right\} \\
			=& \sup_{\psi}\left\{\int\psi d\nu+1+  \int\sup_{\bar x}\sup_q\left (  q -  e^{q}\int  e^{\psi(y)+\bar x\cdot y}\nu(dy) +\bar x\cdot x\right)  \mu(dx)\right\}\\
			=& \sup_{\psi}\left\{\int\psi d\nu+  \int\sup_{\bar x}\left (  -\log  \int  e^{\psi(y)+\bar x\cdot y}\nu(dy) +\bar x\cdot x\right)  \mu(dx)\right\} ,
		\end{align*}
		giving exactly the value of \eqref{eq:main_dual}.
	\end{proof}
	
	We stress that the previous proof is significantly shorter than the proof of the analogous result in the Bass case (see \cite{BaScTs23}). In fact, the above proof can serve as an inspiration for a very short proof of said result in the Bass case. We report this here, for the sake of completeness, but the reader can skip this part without a thought. In fact, the following proof is in the spirit of the so-called Toland's duality; see e.g.\ \cite{Ca08}.
	
	\begin{lemma}[Lemma 3.4  in \cite{BaScTs23}]
		$$\sup_{m\in\MT(\mu,\nu)}\int \MCov(m_x,\gamma)\mu(dx)=\inf_{\alpha\in\P_2(\Rd)}\{\MCov(\nu,\alpha\ast\gamma)-\MCov(\mu,\alpha) \} .$$
	\end{lemma}
	
	\begin{proof}
		By duality for $\MCov$, the r.h.s.\ problem is equal to
		\begin{align*}
			\inf_\alpha\left\{ \inf_{\psi}\left [ \int \psi^*d(\alpha\ast\gamma)+\int\psi d\nu  \right ]- \sup_{A\sim\alpha,X\sim\mu}\mathbb E[A\cdot X] \right\} = & \inf_{\psi} \left\{\inf_{A,X\sim\mu}\left[ \psi^\ast\ast\gamma(A)-A\cdot X \right ] + \int\psi d\nu \right\} \\ = & \inf_{\psi} \left\{-\sup_{A,X\sim\mu}\left[ A\cdot X - \psi^\ast\ast\gamma(A)\right ] + \int\psi d\nu \right\} \\ 
			=& \inf_{\psi} \left\{ \int \psi d\nu- \int (\psi^\ast\ast\gamma)^\ast d\mu \right\},
		\end{align*}
		and this last problem is, thanks to \cite{BaBeScTs23}, equal to the l.h.s.\ problem in the statement of this lemma.
	\end{proof}

	\section{Implications of attainment}
	
	\label{sec:implications}

	In this section, we show that attainment of \eqref{eq:main_dual} is equivalent to attainment of \eqref{eq:vp}, and we provide an explicit correspondence between the respective minimizers. We also prove that, whenever \eqref{eq:main_dual} is attained, the solution to \eqref{eq:introP} can be expressed in terms of the solutions to \eqref{eq:vp} and \eqref{eq:main_dual}. In the whole section, we work under the following assumption.    
	\begin{assumption}\label{ass:int_nu_hyperplane}
		Assume:
		\begin{enumerate}
			\item[\textnormal{(H1)}] The measure $\nu$ has finite exponential moments, that is, $$\displaystyle \int e^{q |y| }\,\nu(dy)<\infty \qquad \text{for all } q \in \R_+;$$
			\item[\textnormal{(H2)}] The measure $\nu$ is not concentrated on any hyperplane.
		\end{enumerate}
	\end{assumption}
	
	Observe that Assumption \textnormal{(H2)} is essentially non-restrictive and only serves to avoid artificial dimensionality. More precisely, if $\nu$ is supported on an affine space $A = x_0 + V$, where $V$ is a linear subspace of $\R^d$ with dimension $r<d$, then $\mu$ is also concentrated on $A$. 
	By choosing $r$ minimal, and up to a change of coordinates, all our results apply in this reduced space.

	\begin{remark}
		Before turning to the proofs, we comment on the question of existence and uniqueness of optimizers for the problems \eqref{eq:introP} and \eqref{eq:main_dual}.
		\begin{enumerate}
			\item First, observe that attainment of \eqref{eq:introP} is straightforward. If the value of \eqref{eq:introP} is $+\infty$, then any martingale transport plan trivially attains it. If the value is finite, then existence of a minimizer follows from the lower semicontinuity of the entropy together with the compactness of $\MT(\mu,\nu)$. Moreover, by strict convexity of $H(\cdot |\mu \otimes \nu)$, the minimizer is unique.
			\item Attainment for \eqref{eq:main_dual} is more delicate and will not be addressed in this article. {As elsewhere in martingale transport, one should be weary of irreducibility and furthermore consider a relaxed version of the dual problem where no integrability of the potential is assumed. Presently dual attainment has been obtained by Nutz and Wiesel \cite{NuWi24} in dimension one, and by Carlier, Malamut and Sylvestre \cite{CaMaSy25} in multiple dimensions, under a strengthening of the irreducibility condition}. In an ongoing work, we study Sinkhorn-type algorithms and prove their convergence under a number of assumptions, which in turn ensures the existence of solutions to both the dual problem and the variational problem. 
			
			Regarding uniqueness, first note that the dual functional $\mathcal D$ is invariant by affine shift: for all admissible $\psi$, $$\mathcal D(\psi + a + b \cdot \id) = \mathcal D (\psi),\quad \text{for all } a,b \in \R^d.$$ Up to this invariance, the optimizer of $\mathcal D$ is unique. This follows from the fact that $\D$ is strictly convex up to affine shift: for all $\psi_1,\psi_2$ admissible, $\D((1-t)\psi_1 + t \psi_2) \leq (1-t) \D(\psi_1) + t \D(\psi_2)$, for all $ t \in ]0,1[$, 
			and equality holds if and only if $\psi_2 - \psi_1$ is affine. 
			Uniqueness (modulo pushing forward by translations by a constant) of an optimizer for \eqref{eq:vp}  follows from our uniqueness statement for \eqref{eq:main_dual}, together with the correspondence between optimizers of \eqref{eq:main_dual} and \eqref{eq:vp} established below.
		\end{enumerate}
	\end{remark}

	\subsection{From \eqref{eq:vp} to \eqref{eq:main_dual}} 
	
	In this subsection, we prove that if \eqref{eq:vp} has finite value and is attained, then \eqref{eq:main_dual} too. We provide an explicit expression of the optimizers of \eqref{eq:main_dual} and \eqref{eq:introP} in terms of the optimizer of \eqref{eq:vp}.\footnote{We thank Aram-Alexandre Pooladian for introducing one of us to the expressions in \eqref{eq:diff_SP}.}

	\begin{proposition}\label{pro:vp_to_D}
		Suppose \eqref{eq:vp} has finite value and $\bar\mu\in\PP_2(\Rd)$ solves \eqref{eq:vp}.  Let $\pi$ be the solution to \eqref{eq:SP}, $(\bar\varphi,\psi)$ the (additive) Schrödinger potentials associated to $(\bar \mu,\nu)$, $T :\bar x \mapsto  \int y d\pi_{\bar x}(y)$ the barycentric map, and $m = (T,\id)_\# \pi$.  
		\begin{enumerate}
			\item The map $\bar\varphi$ is finite, strictly concave, infinitely differentiable, and 
			\begin{equation}\label{eq:diff_SP}
				-\nabla \bar\varphi (\bar x) = T(\bar x) \quad \text{and} \quad \nabla^2 \bar\varphi(\bar x) = -Cov_{\pi_{\bar x}}(Y), \qquad \forall \bar x \in \R^d.
			\end{equation}
			Moreover, $T$ is injective and $T_\# \bar \mu = \mu$.
			\item Let $A$ denote the image of $-\nabla \bar\varphi$ and $h : A \to \R^d$ its inverse. Then $$m_x = \pi_{h(x)} \quad \mu(dx)-a.s.$$  
			\item $m$ attains \eqref{eq:introP};
			\item $\psi$ attains \eqref{eq:main_dual}. Moreover, for $\mu$-a.e. $x \in \R^d$, $h(x)$ is the unique solution to the problem \begin{equation*}\label{eq:inner_dual}
				\sup_{h \in \R^d} \{h \cdot x - \log \int e^{\psi(y)+h\cdot y} \nu(dy)\} \tag{$I_x$}
			\end{equation*}
		\end{enumerate}    
	\end{proposition}

	\begin{proof}
		\textbf{(1)}  By \eqref{eq:SchS1},
		\[
			\bar\varphi(\bar x)=-\log\int e^{\psi(y)+\bar x\cdot y}\nu(dy),
		\]
		so $\bar\varphi$ is finite. The concavity of $\bar\varphi$ follows directly from the Hölder inequality. Define
		\[
			F(\bar x)=\int e^{\psi(y)+\bar x \cdot y}\nu(dy)
		\]
		and note that $\bar\varphi=-\log(F)$. By Equation \eqref{eq:SchS2},
		\[\psi(y) = -\log \int e^{\bar\varphi(\bar x)+\bar x \cdot y} \bar\mu(d \bar x),\]
		so $\psi$ is concave by the Hölder inequality. Moreover, as $-x \cdot y \in L^1(\mu \otimes \nu)$, by \cite[Theorem 4.2]{Nu22},  the Schrödinger potential $\psi$ is $\nu$-integrable and in particular is not identically equal to $+\infty$. Thus, there exists $a,b \in \R^d$ such that $$\psi(y) \leq a + b\cdot y, \qquad \text{for all } y \in \R^d.$$ Therefore, for all $R>0$, $\bar x \in B(0,R)$, and $y \in \R^d$, it holds
		\begin{equation}\label{eq:bound_diff_SP}
			|e^{\psi(y)+\bar x\cdot y}y| \leq |y| e^{|a|+ |b||y|+ R|y|} \leq e^{|a|+(1+|b|+R)|y|}.
		\end{equation}
		As the r.h.s.\ is $\nu$-integrable by Assumption \eqref{ass:int_nu_hyperplane}, $F$ is differentiable and
		\begin{equation*}
			\nabla F (\bar x) = \int e^{\psi(y)+\bar x \cdot y} y \nu(dy).
		\end{equation*}
		Moreover, we have 
		\begin{equation}
			\frac{d \pi_{\bar x}}{d \nu}(y) = \frac{e^{\psi(y)+\bar x \cdot y}}{\int e^{\psi(\tilde{y})+\bar x \cdot \tilde{y}} \nu(d\tilde{y})} ,
		\end{equation}
		so that
		$$\nabla \log F (\bar x) = \frac{\nabla F (\bar x)}{F(\bar x)}  = \frac{\int e^{\psi(y)+\bar x \cdot y} y \nu(dy)}{\int e^{\psi(\tilde{y})+\bar x \cdot \tilde{y}} \nu(d\tilde{y})} = \int y d \pi_{\bar x}(y) = T(\bar x),$$
		which proves the first equality of Equation \eqref{eq:diff_SP}. To compute the Hessian $\nabla^2 \bar\varphi$ of $\bar\varphi$, fix $(i,j) \in \{1,\dots, n\}^2$ and let $Y = (Y_1,\dots,Y_n)$ denote a random variable with law $\pi_{\bar x}$. By the  dominated convergence theorem, whose application can be justified using a bound similar to \eqref{eq:bound_diff_SP}, we have
		\begin{align*}
			-\partial_{ij} \bar\varphi (\bar x) &= \partial_{ij} \log(F(\bar x))\\
			&= \partial_j \left[ \bar x \mapsto \frac{ \int e^{\psi(y)+\bar x\cdot y} y_i \nu(dy)}{\int e^{\psi(y)+\bar x\cdot y} \nu(dy)} \right] \\
			&= \frac{ \int e^{\psi(y)+\bar x \cdot y} y_iy_j \nu(dy) \int e^{\psi(y)+\bar x \cdot y}\nu(dy) -  \int e^{\psi(y)+\bar x \cdot y} y_i \nu(dy) \int e^{\psi(y)+\bar x \cdot y} y_j \nu(dy)}{\left( \int e^{\psi(y)+\bar x \cdot y} \nu(dy)\right)^2}\\
			&= \frac{\int e^{\psi(y)+\bar x \cdot y} y_iy_j \nu(dy)}{\int e^{\psi(y)+\bar x \cdot y} \nu(dy)}- \frac{\int e^{\psi(y)+\bar x \cdot y} y_i \nu(dy)}{\int e^{\psi(y)+\bar x \cdot y} \nu(dy)} \frac{\int e^{\psi(y)+\bar x \cdot y} y_j \nu(dy)}{\int e^{\psi(y)+\bar x \cdot y}  \nu(dy)}.\\
			&= \E_{\pi_{\bar x}}(Y_iY_j)- \E_{\pi_{\bar x}}(Y_i)\E_{\pi_{\bar x}}(Y_j)\\
			&= Cov_{\pi_{\bar x}}(Y_i,Y_j).
		\end{align*}  
		Therefore $-\nabla^2 \bar\varphi (\bar x) = Cov_{\pi_{\bar x}}(Y)$, which proves the second equality in \eqref{eq:diff_SP}. To prove the strict concavity of $\bar\varphi$, it is now sufficient to prove that $Cov_{\pi_{\bar x}}(Y)$ is positive definite  for all $\bar x \in \R^d$. Assume by contradiction there exists $ x \in \R^d \setminus \{0\}$ such that $ x^\top Cov_{\pi_{\bar x}}(Y)  x = 0$, i.e., $Cov_{\pi_{\bar x}}(x \cdot Y) =0$. Then $x\cdot Y$ would be constant almost surely, that is, there exists $p \in \R$ such that $x \cdot Y = p$. Therefore $\pi_{\bar x}(\{y~;~x\cdot y = p\})=1$ and $\pi_{\bar x}$ would be concentrated on a hyperplane. As $\psi$ is $\nu$-integrable, it is also $\nu$-a.e.\ finite, so that $\frac{d\pi_{\bar x}}{d\nu}(y) \propto e^{\psi(y)+ \bar x \cdot y} \in ]0,+\infty[$ for $\nu$-a.e. $y$, which implies $\pi_{\bar x}$ is equivalent to $\nu$. In particular $\nu$ would charge a hyperplane, which is a contradiction and proves that $Cov_{\pi_{\bar x}}(Y)$ is positive definite and the strict concavity of $\bar\varphi$. As $T = -\nabla \bar\varphi$ and gradient of strictly convex functions are injective, $T$ is injective. We finally prove that $T_\# \bar \mu = \mu$, by using the first order conditions of $\bar{\mu}$.\\
		\textbf{First-order conditions of $\bar{\mu}$:} Let  $(\bar X, X)$ be a coupling maximizing the covariance between $\bar \mu$ and $\mu$, fix $n \geq 1$, and define $$w : (x,\bar x) \mapsto 1_{|x-T(\bar x)|^2 \leq n, |x| \leq n,|\bar x| \leq n} (x-T(\bar x)).$$ For all $u >0$, define $\bar \mu_u := \Law(\bar X + u \cdot w(\bar X,X))$.
		Writing $\mathcal F : \bar \mu \mapsto SP(\bar \mu,\nu) + \MCov(\bar \mu,\mu)$, we show that 
		\begin{equation}\label{eq:rhs_der_F}
			\lim_{u \to 0^+} \frac1u(\mathcal{F}(\bar\mu_u)- \mathcal{F}(\bar \mu) )= \E\left(w(\bar X,X)\cdot (X - T(\bar X))\right).
		\end{equation}
		As $\bar X $ has a second moment and $w$ is bounded, $\bar\mu_u \in \PP_2(\R^d)$ and so $\mathcal{F}(\bar \mu) \geq \mathcal{F}(\bar \mu_u)$. As $w(\bar x, x)(x-T(\bar x)) \geq 0$, we get 
		\begin{equation*}
			\limsup_{u \to 0^+} \frac1u(\mathcal{F}(\bar\mu_u)- \mathcal{F}(\bar \mu) ) \leq 0 \leq \E(w(\bar X, X)\cdot (X-T(\bar X))).
		\end{equation*}
		To prove the $\liminf$ inequality, as $MC(\bar \mu, \mu) = \E(X\cdot \bar X)$, we have
		\begin{equation*}
			\MCov(\bar \mu_u,\mu)- \MCov(\bar \mu,\mu) \geq \E((\bar X + u w(\bar X, X))\cdot X) - \E(\bar X \cdot X) = u\E(w(\bar X,X)\cdot X),
		\end{equation*}
		so
		\begin{equation*}
			\liminf_{u \to 0^+} \frac{\MCov(\bar \mu_u,\mu)-\MCov(\bar \mu,\mu)}{u} \geq \E(w(\bar X,X)\cdot X).
		\end{equation*}
		It is therefore sufficient to show  $$\liminf_{u \to 0^+}  \frac{SP(\bar \mu_u,\nu)-SP(\bar \mu,\nu)}{u} \geq -\E(w(\bar X,X)\cdot T(\bar X)).$$  Using the standard primal/dual equality in entropic optimal transport (see e.g. \cite[Theorem 4.7]{Nu22}) and the identity $\bar\varphi(\bar x)=-\log\int e^{\psi(y) + \bar x \cdot y} \nu(dy)$, equivalently $e^{\bar\varphi(\bar x)}\int e^{\psi(y) + \bar x \cdot y}\nu(dy)=1$, which holds for all $\bar x \in \R^d$, we have
		\begin{align*}
			SP(\bar \mu_u,\nu)- SP(\bar \mu,\nu) &\geq \int \bar\varphi(\bar x) d (\bar \mu_u - \bar \mu)(\bar x) + \int \psi(y) d (\nu - \nu)(y) \\
			&+  \log \int e^{\bar\varphi(\bar x)}\int  e^{ \psi(y) + \bar x \cdot y} \nu(dy) \bar\mu_u(d\bar x) \\
			&- \log \int e^{\bar\varphi(\bar x)}\int  e^{   \psi(y) + \bar x \cdot y} \nu(dy) \bar\mu(d\bar x)\\
			&= \int \bar\varphi(\bar x) d (\bar \mu_u - \bar \mu)(\bar x)\\
			&= \E(\bar\varphi(\bar X+uw(\bar X,X))-\bar\varphi(\bar X)),
		\end{align*}
		so
		\begin{equation*}
			\liminf_{u \to 0^+} \frac{SP(\bar \mu_u,\nu)- SP(\bar \mu,\nu)}{u} \geq \liminf_{u\to 0^+} \E\left( \frac{\bar\varphi(\bar X+uw(\bar X,X))-\bar\varphi(\bar X)}{u}\right).
		\end{equation*}
		Since $\bar\varphi$ is infinitely differentiable, $$  u^{-1}(\bar\varphi(\bar x+uw(\bar x,x))-\bar\varphi(\bar x))= \nabla \bar\varphi (\bar x) \cdot w(x,\bar x) + \frac{u}{2} w(x,\bar x)^\top \nabla^2 \bar\varphi(\bar x) w(x,\bar x) + o(u).$$ As $w$ is supported on a compact set and $\nabla \bar\varphi$ and $\nabla^2 \bar\varphi$ are continuous, the  dominated convergence theorem applies and
		$$ \liminf_{u\to 0^+} \E\left( \frac{\bar\varphi(\bar X+uw(\bar X,X))-\bar\varphi(\bar X)}{u}\right) = \E(\nabla \bar\varphi (\bar X) \cdot w(X,\bar X))= -\E(T(\bar X)\cdot w(X,\bar X)).$$ Therefore \eqref{eq:rhs_der_F} holds, 
		and the above bounds further yield $$\E\left(1_{|X|\leq n, |\bar X | \leq n, |X-T(\bar X)|^2 \leq n}\cdot |X - T(\bar X)|^2\right)=\E\left(w(\bar X,X)\cdot (X - T(\bar X))\right) = 0.$$ By monotone convergence, $\E(|X-T(\bar X)|^2) = 0$, so $X = T(\bar X)$ almost surely and $T_\# \bar \mu = \mu$.
		\newline
		\textbf{(2)} Given $a,b :\R^d \to \R$ continuous bounded, as $m = (T,\id)_\# \pi$ and $T_\# \bar \mu = \mu$, we get 
		\begin{align*}
			\int a(T(\bar x))\int b(y) \pi_{\bar x}(d y) \bar \mu(d\bar x) &= \int a(x) b(y) [(T,\id)_\#\pi](dx,dy)
			= \int a(x) b(y) m(dx,dy)\\
			&= \int a(x) \int b(y) m_x(dy) \mu(dx)
			= \int a(T(\bar x)) \int b(y) m_{T(\bar x)} \bar\mu(d \bar x). 
		\end{align*}
		Using $T_\# \bar \mu = \mu$ again, and the fact that $\bar x = h(T (\bar x))$ for $\bar \mu$-a.e. $\bar x$, it holds
		\begin{align*}
			0 = \int a(T(\bar x)) \int b(y) (m_{T(\bar x)}-\pi_{h(T(\bar x))})(d y) \bar \mu(dx)
			&=  \int a(x) \int b(y) (m_{x}- \pi_{h(x)})(d y)  \mu(dx).
		\end{align*}
		Since this holds for all continuous bounded maps $a$, then 
		$\int b(y) (m_{x}- \pi_{h(x)})(d y)$ ($\mu(dx)-a.s.$),  
		which implies $m_x = \pi_{h(x)}$ for $\mu$-a.e.\ $x$.
		\newline
		\textbf{(3)} As $T$ is the gradient of a convex function and a transport map from $\bar \mu$ to $\mu$, $$\MCov(\bar \mu,\mu) = \int \bar x \cdot T(\bar x) \bar \mu( d \bar x) = \int \bar x \cdot \int y \pi_{\bar x}(dy) \bar \mu(d \bar x) = \int \bar x \cdot y \  \pi(dx,dy).$$ Therefore,
		\begin{align*}
			H(m|\mu \otimes \nu) &= \int H(m_x|\nu) \mu(dx) = \int H(\pi_{h(x)}|\nu) \mu(dx)
			= \int H(\pi_{h(T(\bar x))}|\nu) \bar \mu(d \bar x) 
			= \int H(\pi_{\bar x}|\nu) \bar \mu(d \bar x)\\ &= H(\pi|\bar \mu \otimes \nu)
			= H(\pi|\bar \mu \otimes \nu) - \int \bar x \cdot T(\bar x) \bar \mu( d \bar x) + \MCov (\bar \mu,\mu) = value\eqref{eq:vp} = value\eqref{eq:introP}.
		\end{align*}
		\newline
		\textbf{(4)} Given a probability measure $p \ll \nu $ and a measurable function $f$ such that $\int e^f ~d\nu < + \infty$, the Gibbs variational principle , see \cite[Appendix A, (A.5)]{Le14}),  tells us that 
		$$ H(p\mid \nu) \geq \int f dp - \log  \int e^f d\nu,$$
		with equality if and only if $\frac{dp}{d\nu}(y) = \frac{e^f}{\int e^f d\nu}$. Denoting by $x$ the barycentre of $p$ and applying the previous inequality with $f(y) = \psi(y) + h\cdot y$ for all $h \in \R^d$, we get $$H(p|\nu) - \int \psi dp \geq \sup_{h \in \R^d} \left\{h\cdot x - \log \int e^{h\cdot y + \psi(y)} d\nu(y) \right\},$$ and if there exists $h^* \in \R^d$ such that $\frac{dp}{d\nu}(y) = \frac{e^{h^*\cdot y + \psi(y)}}{\int e^{h^*\cdot \tilde{y}+ \psi(\tilde{y})} d\nu(\tilde{y})}$, the inequality becomes an equality and the supremum is attained for $h=h^*$. Now, according to the second point of this lemma, for $\mu$-a.e. $x \in \R^d$, $$\frac{d m_x}{d\nu}(y) = \frac{d \pi_{h(x)}}{d\nu}(y) = \frac{e^{\psi(y) + h(x)\cdot y}}{\int e^{\psi( \tilde y) + h(x) \cdot \tilde y  } \nu(d \tilde{y})}.$$  By the above argument, for $\mu$-a.e. $x \in \R^d$,	$$H(m_x|\nu) - \int \psi(y) dm_x(y) = \sup_{h \in \R^d} \left\{ h \cdot x - \log\int e^{\psi(y)+h \cdot y} \nu(dy)\right\}.$$ Integrating along $\mu$ and using $m \in \MT(\mu,\nu)$, we get
		\begin{align*}
			value\eqref{eq:main_dual} &= value\eqref{eq:introP} = \int H(m_x|\nu) \mu(dx)\\ &= \int \psi(y) d\nu(y) + \int \sup_{h \in \R^d} \left\{ h \cdot x - \log\int e^{\psi(y)+h \cdot y} \nu(dy)\right\}d \mu(x).
		\end{align*}
		Fix $x \in A = T(\R^d)$. For all $x \in A$, it holds $T(h(x)) = x$. As  $-\bar\varphi$ is convex and differentiable, $\tilde h$ solves \eqref{eq:inner_dual} if and only if $x \in \partial (-\bar\varphi)(\tilde h) $ if and only if $x = -\nabla \bar\varphi(\tilde h) = T(\tilde h) $ if and only if  $\tilde h = h(x)$. 
	\end{proof}

	\subsection{From \eqref{eq:main_dual} to \eqref{eq:vp}}
	
	In this subsection, we prove that if \eqref{eq:main_dual} has finite value and is attained, then \eqref{eq:vp} also is. In doing so, we provide an explicit expression of the optimizers of \eqref{eq:vp} and \eqref{eq:introP} in terms of the optimizer of \eqref{eq:main_dual}. In addition to Assumption \eqref{ass:int_nu_hyperplane}, we assume the following.
	
	\begin{assumption}\label{ass:exp_affine_unique}
		Assume:
		\begin{enumerate}
			\item[\textnormal{(H3)}] The problem $\eqref{eq:main_dual}$ has finite value and is attained by a $\nu$-integrable map $\psi$ such that there exists $(a,b) \in \R_+^2$ with $|\psi(y)| \le a+b|y|$ for $\nu$-a.e.\ $y$;
			\item[\textnormal{(H4)}]\label{H4} For $\mu$-a.e. $x$, the problem \ref{eq:inner_dual}	admits a unique solution, denoted $h(x)$.
		\end{enumerate}
	\end{assumption}
	
	\begin{remark}
		These two assumptions on \(\psi\) and on \eqref{eq:inner_dual}, provide \emph{necessary and sufficient} conditions for the attainment of \eqref{eq:vp} by $\bar \mu \in \PP_2(\R^d)$. Indeed, we saw in the proof of Proposition \ref{pro:vp_to_D} that if \eqref{eq:vp} is attained by such $\bar \mu$, then \eqref{eq:main_dual} is attained by the second Schrödinger potential with respect to $(\bar \mu,\nu)$, and we proved that this potential satisfies the bound in $(H_3)$. Moreover, we saw in the proof that \eqref{eq:inner_dual} has a unique solution given by $(-\nabla \bar\varphi)^{-1} (x)$ where $\bar\varphi$ is the first Schrödinger potential with respect to $(\bar \mu,\nu)$. So these conditions are necessary to prove attainment of \eqref{eq:vp}. 
	\end{remark}
	
	To derive the first order conditions associated to the optimality of $\psi$, one needs to be able  to differentiate maps defined as supremum. This will require the following version of the Danskin lemma, which follows directly from \cite[Corollary~4]{MiSe02}.
	
	\begin{lemma}[Danskin Lemma]\label{lem:danskin_local}
		Consider a continuous map $\Phi:\R\times \R^d \to\R$  such that:
		
		\begin{enumerate}
			\item\label{point:Danskin_loc}  There exist $\epsilon_0> 0$ and a compact set $K \subset \R^d$ such that, for all $\epsilon \in (-\epsilon_0,\epsilon_0)$,
			\[
			\sup_{h\in \R^d}\Phi(\epsilon,h)=\sup_{h\in K}\Phi(\epsilon,h);
			\]
			\item\label{point:Danskin_diff} For all $h\in K$, the map $\epsilon\mapsto \Phi(\epsilon,h)$ is differentiable on $(-\epsilon_0,\epsilon_0)$;
			\item\label{point:Danskin_C_1} The map $(\epsilon,h)\mapsto \partial_\epsilon\Phi(\epsilon,h)$ is continuous on $(-\epsilon_0,\epsilon_0) \times K$;
			\item\label{point:Danskin_uniqueness} $\Phi(0,\cdot)$ admits a unique maximizer $h_0$ on $K$.
		\end{enumerate}
		Then, the map $V(\epsilon):=\sup_{h\in \R^d}\Phi(\epsilon,h)$ is differentiable at $0$, and	$V'(0)=\partial_\epsilon\Phi(0,h_0).$
	\end{lemma}
	
	Given a set $C \subset \R^d$, we denote by $\conv(C)$, $\aff(C)$ and $\textnormal{ri}(C)$ its convex hull, its affine hull and its relative interior, respectively. Given a map $f$, we denote by $\dom(f)$ its domain, by $f^*$ its Fenchel-Legendre transform, and by $\partial f (x)$ its sub-differential in $x$. 
	
	\begin{lemma}\label{lem:application_Danskin}
		Let $(x ,h ,\epsilon) \in \R^d \times \R^d \times \R$, and consider a continuous map $\eta: \R^d \to \R$ with compact support. Define
		\[
		Z_\epsilon(h)
		=
		\int_{\R^d}\exp\big(\psi(y)+h\cdot y+\epsilon\,\eta(y)\big)\,\nu(dy),
		\qquad
		\Phi_x(\epsilon,h)
		=
		h\cdot x-\log Z_\epsilon(h).
		\]
		and 
		\[
		V_x(\epsilon)
		=
		\sup_{h\in\R^d}\Phi_x(\epsilon,h).
		\]
		Then the following assertion are satisfied.
		\begin{enumerate}
			\item The map  $\Phi_x$ is finite and continuous.
			\item\label{point:formul_diff} For all $h\in\R^d$, the map
			$\epsilon\mapsto \Phi_x(\epsilon,h)$ is differentiable on $\R$, and, for every $(\epsilon,h) \in \R \times \R^d$,
			\[
			\partial_\epsilon\Phi_x(\epsilon,h)
			=
			-\int \eta(y)\,m^{\epsilon,h}(dy),
			\]
			where
			\[
			m^{\epsilon,h}(dy)
			:=
			\frac{e^{\psi(y)+h\cdot y+\epsilon \eta(y)}}
			{\int e^{\psi(\tilde y)+h\cdot \tilde y+\epsilon \eta(\tilde y)}\,\nu(d \tilde y)}\,\nu(dy).
			\] 
			\item The map
			$(\epsilon, h) \mapsto \partial_\epsilon\Phi_x(\epsilon,h)$ is continuous on $\R\times\R^d$.
			\item
			Set $\Gamma = \Phi_x(0,\cdot)$ and $m^{0,h}=m^h.$
			\begin{enumerate}
				\item\label{point:Z_0_diff} The map $\Gamma$ is differentiable and for all $h$, $\nabla \Gamma (h) =  \int y m^h(dy)$;
				\item\label{point:gradient_in_rel_int} For all $h \in \R^d$, $\nabla \Gamma(h)\in \mathrm{ri}\big(\dom  \Gamma^*\big)$;
				\item\label{point:coercitivity_on_affine_span} If $x\in \mathrm{ri}(\dom \Gamma^*)$, then  $\Gamma$ is coercive along $V=\aff(\dom \Gamma^*)-x$, i.e.,
				$$\lim_{ h\in V, |h| \to +\infty} \Gamma(h) = -\infty;$$
				\item\label{point:affine_is_full} $\aff(\dom \Gamma^*) = \R^d$;
				\item\label{point:full_coercitvity} If \eqref{eq:inner_dual} admits a unique solution, then the map $\Gamma$ is coercive,  i.e., $$\lim_{|h| \to +\infty} \Gamma(h) = -\infty.$$
			\end{enumerate}
			
			\item $V_x$ is differentiable at $\epsilon=0$ and
			\[
			\left.\frac{d}{d\epsilon}\right|_{\epsilon=0}
			V_x(\epsilon)
			=
			-\int \eta(y)\,m^{h(x)}(dy).
			\]
		\end{enumerate}
	\end{lemma}
	
	\begin{proof}
		\textbf{(1)} As $\psi$ is $\nu$-integrable, it has finite value $\nu(dy)$-a.e., so $Z_\epsilon(h)$ is positive. By Assumption \ref{ass:exp_affine_unique}, $$|\psi(y)+h\cdot y+\epsilon\,\eta(y)| \leq (a + \epsilon |\eta|_\infty) + (b+ |h|)|y|,$$ and the r.h.s. is $\nu$-integrable, which implies $Z_\epsilon(h) < +\infty$. This proves $\Phi_x$ is finite. Continuity follows by the dominated convergence theorem, which application is justified by the above bound and Assumption \ref{ass:exp_affine_unique}.\\
		\textbf{(2)} 
		Fix $h \in \R^d$ and $\epsilon_0 \in \R$. For  $|\epsilon - \epsilon_0| \leq 1$, we have by Assumption \ref{ass:exp_affine_unique} $$|\eta(y)| e^{\psi(y)+h\cdot y+\epsilon\,\eta(y)} \leq |\eta|_\infty   e^{a +(|\epsilon_0| + 1) |\eta|_\infty + (b+ |h|)|y|}.$$ By Assumption \ref{ass:int_nu_hyperplane} the r.h.s. is $\nu$-integrable, and we conclude by the dominated convergence theorem and by the arbitrarity of $\epsilon_0$.\\
		\textbf{(3)} Let $(\epsilon_n,h_n) \to (\epsilon,h)$.
		For all $M,\delta>0$ s.t. $h_n,h \in B(0,M)$ and $\epsilon_n,\epsilon \in B(0, \delta)$, it holds
		\[|\eta(y)|e^{\psi(y)+h_n \cdot y + \epsilon_n \eta(y)} \leq |\eta|_\infty e^{a + \delta|\eta|_\infty+  (b+ M)|y|} \; \; \text{and}\; \; e^{\psi(y)+h_n \cdot y+\epsilon_n \eta(y)} \leq  e^{a +\delta|\eta|_\infty+  (b+ M)|y|}.\] As both r.h.s. are $\nu$-integrable, the dominated convergence theorem can be applied to prove that the numerator and the denominator appearing in the expression of $\partial_\epsilon \Phi_x$ are continuous, which concludes.\\
		\textbf{(4)} We successively prove the five points. Set $\tilde{\nu} = e^{\psi}\cdot \nu$.
		\begin{enumerate}
			\item[(a)] According to Assumption \ref{ass:exp_affine_unique}, for any given $M>0$ and $h \in B(0,M)$,
			\begin{equation*}
				|e^{h \cdot y + \psi(y)}y| \leq |y| e^{M|y| + a|y|+b} \leq e^{(1+a+M)|y| + b},
			\end{equation*}
			and the r.h.s. is $\nu$-integrable. By the dominated convergence theorem, the map $J: h \mapsto \int e^{y \cdot h + \psi(y)} \nu(dy)$ is differentiable and $$ \nabla J (h) =  \int e^{y \cdot h + \psi(y)} y\nu(dy).$$ Therefore $\Gamma$ is differentiable and $$\nabla \Gamma(h) = \frac{\int e^{y \cdot h + \psi(y)} y\nu(dy)}{\int e^{y \cdot h + \psi(y)} \nu(dy)} = \int y m^h(dy).$$
			\item[(b)] Fix $h\in\R^d$ and set $x=\nabla\Gamma(h)$. As moment generating functions are convex $\Gamma$ is convex, so it holds
			\[
			\Gamma^*(x)=\sup_{y\in\R^d}\{y\cdot x-\Gamma(y)\} = h\cdot x-\Gamma(h)<\infty,
			\]
			and therefore $x\in\dom\Gamma^*$. Assume by contradiction
			$x\notin \mathrm{ri}(\dom\Gamma^*)$. As $\dom\Gamma^*$ is convex, we know (see e.g. \cite[Lemma 4.2.1]{HiLe93}) that there exists a supporting
			hyperplane at $x$, that is, there exists $u\in\R^d\setminus\{0\}$ and $\alpha\in\R$ such that
			\begin{equation}\label{eq:support}
				u\cdot x=\alpha \qquad \text{and} \qquad u\cdot z\le \alpha \quad\text{for all } z \in\dom\Gamma^*.
			\end{equation}
			Define $g : t \in \R \mapsto \Gamma(h+t u).$ Since \(\Gamma\) is convex and differentiable, \(g\) is convex and differentiable, and by the chain rule $	g'(t)=u\cdot \nabla \Gamma(h+t u).$ By Fenchel--Young equality,
			\[
			\Gamma^*(\nabla\Gamma(h+t u))
			=(h+t u)\cdot \nabla\Gamma(h+t u)-\Gamma(h+t u)<\infty, \quad \text{for all } t \in \R.
			\]
			Therefore \(\nabla\Gamma(h+t u)\) belongs to \(\dom\Gamma^*\) and by Equation \eqref{eq:support},
			\[
			g'(t)=u\cdot \nabla\Gamma(h+t u)\le \alpha
			\qquad\text{for all } t\in\R.
			\]
			As $g'(0)=u\cdot \nabla\Gamma(h)=u\cdot x=\alpha$ and \(g'\) is nondecreasing by convexity of $g$, $g'$ is constant equal to $\alpha$ on $[0,+\infty[$.
			Thus, for all $t \geq 0$
			\[
			\log\int e^{t\,u\cdot y}\,m^h(dy) = \Gamma(h+t u)-\Gamma(h)= \int^t_0 g'(t) dt = \alpha t,
			\]
			which rewrites
			\[
			\int e^{t(u\cdot y-\alpha)}\,m^h(dy)=1
			\qquad\text{for all } t\ge 0.
			\]
			Let \(Y\sim m^h\) and define \(X=u\cdot Y-\alpha\). By the previous point,
			$$\E[X]	=u\cdot \int y\,m^h(dy)-\alpha = u\cdot x-\alpha =0,$$
			and the previous identity reads $\E[e^{tX}]=e^{t\E[X]}.$ In particular $\E[e^{X}]=e^{\E[X]}$. 	Since the exponential map is strictly convex, equality in Jensen's inequality implies that \(X\) is almost surely constant. Since \(\E[X]=0\), we conclude that \(X=0\) almost surely, that is,	$u\cdot y=\alpha$ for $m^h$-a.e. $y$. Since \(m^h\sim \tilde\nu \sim\), it follows that $u\cdot y=\alpha$ for $\nu$-a.e.$y$. This contradicts the fact that \(\nu\) is not supported on any affine hyperplane. Therefore $	x\in \ri(\dom\Gamma^*).$
			\item[(c)] Since $x\in \mathrm{ri}(\dom\Gamma^*)$, there exists $r>0$ such that
			\begin{equation}\label{eq:ri_ball_affine}
				x+ [B(0,r)\cap V] \subset \mathrm{ri}(\dom\Gamma^*).
			\end{equation}
			Define $S_V=\{u\in V:\ |u|=1\}.$
			By \eqref{eq:ri_ball_affine}, the compact set $K :=\{x+ru~;~ u\in S_V\}$ is a subset of $\mathrm{ri}(\dom\Gamma^*)$. Since $\Gamma^*$ is a convex function, it is continuous on $\mathrm{ri}(\dom\Gamma^*)$ (see \cite[Remark 3.1.3]{HiLe93}), and thus bounded above on every
			compact subset of $\mathrm{ri}(\dom\Gamma^*)$. Hence $
			M:=\sup_{z\in K}\Gamma^*(z)< +\infty$. Now, fix $h\in V \setminus\{0\}$, and write $h=t u$ with $t=|h|$ and $u=h/|h|\in S_V$.	By definition, we have
			\[
			\Gamma^*(x+ru)\ge (x+ru)\cdot(tu)-\Gamma(tu)=t(x\cdot u+r)-\Gamma(tu).
			\]
			Rearranging yields
			$
			\Gamma(tu)\ge t(x\cdot u+r)-\Gamma^*(x+ru)
			$ 
			and therefore
			\[
			\Phi_x(0,h)=t\,x\cdot u-\Gamma(tu)\le -rt+\Gamma^*(x+ru)\le -r|h|+M.
			\]
			Thus $\lim_{h \in V, |h | \to \infty} \Phi_x(0,h) = -\infty$, which is the
			claimed coercivity along $V$.
			\item[(d)] Assume by contradiction that $\aff(\dom \Gamma^*)\neq \R^d$, meaning that there exist
			$u\in \R^d\setminus\{0\}$ and $\alpha\in\R$ such that
			$u\cdot z=\alpha$ for all  $z\in \dom \Gamma^*.$ Now define
			\[
			\lambda(t)=\Gamma(tu)=\log \int e^{t\,u\cdot y}\,\tilde\nu(dy), \qquad t \in \R.
			\]
			Since \(\Gamma\) is differentiable, \(\lambda\) is too, and as $\Gamma(\R^d) \subset \ri(\dom \Gamma^*) \subset \dom \Gamma^*$,
			\[
			\lambda'(t)=u\cdot \nabla\Gamma(tu)=\alpha \qquad \text{for all } t\in\R.
			\]
			Therefore \(\lambda(t)=\alpha t+\beta\) for some \(\beta\in\R\), i.e.
			\[
			\int e^{t\,u\cdot y}\,\tilde\nu(dy)=e^\beta e^{\alpha t}\qquad \text{for all } t\in\R.
			\]
			Evaluating at \(t=0\), we get \(e^\beta= \tilde \nu(\R) \), so
			\[
			\int e^{t(u\cdot y-\alpha)}\,\left[\frac{\tilde \nu}{\tilde \nu(\R)}\right](dy)= 1\qquad \text{for all } t\in\R.
			\]
			Reasoning as in the end of Point \ref{point:gradient_in_rel_int}, this implies that $\tilde \nu$ and therefore $\nu$ is concentrated on a hyperplane. This contradicts Assumption \ref{ass:int_nu_hyperplane} and concludes.
			\item[(e)] As $h(x)$ is the unique solution of \eqref{eq:inner_dual}, $x = \nabla \Gamma (h(x))$. Thus, by Point \ref{point:gradient_in_rel_int}, $x \in \textnormal{ri}(\textnormal{dom} \Gamma^*)$. Then, by Point \ref{point:coercitivity_on_affine_span}, $\Phi_x(0,\cdot)$ is coercive on $\aff(\dom\Gamma^*)-x$, and by Point \ref{point:affine_is_full} $\aff(\dom\Gamma^*)-x = \R^d$. 
		\end{enumerate}
		
		\textbf{(5)} For all $\epsilon > 0$ and $h \in \R^d$, we have 
		$ e^{-\epsilon |\eta|_\infty}Z_0(h) \leq Z_\epsilon(h) \leq e^{\epsilon |\eta|_\infty}Z_0(h)$. Therefore, $|\Phi_x(\epsilon,h)-\Phi_x(0,h)|
		\le
		|\eta|_\infty\,|\epsilon|$. Thus, for all $|\epsilon| \leq 1$ and $h$ such that $\Phi_x(\epsilon,h) \geq \Phi_x(\epsilon,h_0)$ it holds 
		\begin{align*}
			\Phi_x(\epsilon,h)
			\ge \Phi_x(\epsilon,h_0)
			\ge \Phi_x(0,h_0)-|\eta|_\infty,
		\end{align*}
		that is, $h$ belongs to $K:=\{\Phi_x(0,\cdot) \geq M\}$, with $M := \Phi_x(0,h_0) -  | \eta |_\infty$. We get,  $$\sup_{h\in \R^d}\Phi(\epsilon,h)= \sup_{h\in \R^d ; \Phi(\epsilon,h) \geq \Phi(\epsilon,h_0) }  \Phi(\epsilon,h) = \sup_{h\in K}\Phi(\epsilon,h).$$ By coercivity of $\Phi_x(0,\cdot)$, the set $K$  is bounded. Since  $\Phi_x(0,\cdot)$ is continuous $K$ is in fact compact. 
		
		We now verify the assumptions of Lemma \ref{lem:danskin_local} for $\Phi = \Phi_x$: the compact reduction follows from the above coercivity argument, the second and thirds points of the present lemma give differentiability in $\epsilon$ and joint continuity of $\partial_{\epsilon}\Phi_x$ on  $(-\epsilon_0,\epsilon_0)\times K$, finally, Assumption \eqref{ass:exp_affine_unique} gives uniqueness of the maximizer with $h_0 = h(x)$. Therefore, Lemma \ref{lem:danskin_local} yields $V_x'(0)=\partial_\epsilon \Phi_x(0,h(x)) = - \int \eta(y) m^{h(x)}(dy)$.
	\end{proof}
	
	We now prove the main result of this subsection.

	\begin{proposition}\label{pro:from_D_to_VP}
		Define 
		\begin{enumerate}
			\item $\bar \mu = h_\# \mu$;
			\item  $\bar\varphi(\bar x) = -\log \int e^{\psi(y)+\bar x \cdot y} \nu(dy)$;
			\item $\frac{dm_x}{d\nu}(y) = \frac{e^{h(x)\cdot y + \psi(y)}}{\int e^{h(x)\cdot y + \psi(y)} \nu(dy)}$ and $m(dx,dy) = \mu(dx)m_x(dy)$;
			\item  $\frac{d \pi}{d \bar \mu \otimes  \nu}(\bar x,y) =  \exp\left( \bar\varphi (\bar x) + \psi(y) + \bar x \cdot y\right)$.
		\end{enumerate}
		Then:
		\begin{enumerate}
			\item $\bar \mu$ solves \eqref{eq:vp}, $\pi$ solves $\eqref{eq:SP}$, and $(\bar\varphi,\psi)$ are the corresponding Schrödinger potential associated to $(\bar \mu,\nu)$;
			\item $m$ solves \eqref{eq:introP}.
		\end{enumerate}
	\end{proposition}
	
	\begin{proof}
		\textbf{First order condition of $\psi$ :} We first derive the first order condition associated to $\psi$ to show that $m$ has second marginal $\nu$. Set $Z(h) = \int e^{h \cdot y + \psi(y)} \nu(dy)$, $\Gamma(h) = \log(Z(h))$, and $$\mathcal{D}(\psi) = \int \psi(y) \nu(dy) + \int \Gamma^*(x) \mu(dx).$$ Consider a continuous map $\eta: \R^d \to \R$ with compact support and set $\psi_\epsilon = \psi + \epsilon \eta$. As $\psi$ solves \eqref{eq:main_dual}, we know that $\mathcal{D}$ admits a maximum at $\psi$, so $r : \epsilon \mapsto \mathcal{D}(\psi_\epsilon)$ admits a maximum at $0$. Therefore, if $r$ is differentiable at $0$, $r'(0) = 0$. We now prove that $r$ is differentiable at $0$ and compute its derivative at $0$. First, the map $r_1 : \epsilon \mapsto \int \psi_\epsilon d \nu$ is clearly differentiable and $r_1'(0) = \int \eta(y) \nu(dy)$. To differentiate $r_2 := r - r_1$, for all $x$ such that \eqref{eq:inner_dual} admits a unique solution, define the map $\Phi_x : \R \times \R^d \to \R$ by $$\Phi_x(\epsilon,h) = h \cdot x - \log \int e^{h \cdot y + \psi(y) + \epsilon \eta(y)}\nu(dy)$$ and the map
		$V_x : \R \to \R$ by $V_x(\epsilon) = \sup_{h \in \R^d} \Phi_x(\epsilon,h)$. 
		By Lemma \ref{lem:application_Danskin}, $V_x$ is differentiable at 0 and  
		\begin{equation*}
			V_x'(0) = - \int \eta(y) m_x(dy). 
		\end{equation*}
		As $\left| - \int \eta(y) m_x(dy) \right| \leq |\eta|_\infty$, using the dominated convergence theorem again, we get that $r_2 : \epsilon \mapsto \int V_x(\epsilon) \mu(dx)$ is differentiable at $0$, and $r_2'(0) = \int \left( - \int \eta(y) m_x(dy) \right) \mu(dx).$ Therefore, $r = r_1 + r_2$ is differentiable at $0$ and $0 = r'(0) = r_1'(0) + r_2'(0)$, that is,
		\begin{equation*}\label{eq:m_second_marg_nu}
			\int \eta(y) \nu(dy) = \iint  \eta(y) m_x(y) \nu(dy).
		\end{equation*}
		Therefore $m \in \Cpl(\mu,\nu)$.
		
		\vspace{0.5cm}
		\noindent\textbf{First order condition of $h$:} Let $x$ be such that the problem \eqref{eq:inner_dual} has a unique solution $h(x)$, so that $x \in \partial \Gamma(h(x))$. By Lemma \ref{lem:application_Danskin}, $\Gamma$ is differentiable in $h(x)$ and $\nabla \Gamma (h(x)) = \int y m^{h(x)}(dy)$. By convexity of $\Gamma$, $\int y m_x(dy) = x$. As this holds for $\mu$-a.e. $x$, $m \in \MT(\mu,\nu)$.
		\vspace{0.5cm}
		
		\noindent\textbf{Proof of (1):} For all $\theta :\R^d \to \R$ continuous bounded, by the definition of $\bar\varphi$, we have 
		\begin{align*} 
			\int_{\R^d} \theta(\bar x) d\pi(\bar x,y) &= \int \theta(\bar x)  e^{\bar\varphi(\bar  x)} \int e^{\psi(y)+\bar x \cdot y} \nu(dy) \bar\mu(d \bar x)= \int \theta(\bar x) \bar{\mu}(d \bar x).
		\end{align*}
		Moreover,
		\begin{align*}
			\int \theta(y) d\pi(\bar x,y) &= \iint  \theta(y) e^{\bar\varphi(\bar x)+\psi(y)+\bar x \cdot y} \nu(dy) \bar\mu(d\bar x) = \iint  \theta(y) e^{\bar\varphi(h(x))+\psi(y)+h(x) \cdot y} \nu(dy) \mu(dx) \\
			&= \iint \theta(y) m_x(dy) \mu(dx) = \int \theta(y) \nu(dy). 
		\end{align*}
		Therefore $\pi \in \Cpl(\bar \mu,\nu)$. The usual characterization result of the entropic minimizer (see \cite[Theorem 4.2, b)]{Nu22}) proves that $\pi$ solves $\eqref{eq:SP}$ and $(\bar\varphi,\psi)$ are the additive Schrödinger potentials with respect to $(\bar \mu,\nu)$. To prove that $\bar \mu$ solves \eqref{eq:vp} it is sufficient to prove that $\mathcal{D}(\psi) = SP(\bar \mu,\nu) + \MCov(\bar \mu ,\mu).$ As $\bar\varphi$ and $\psi$ are the additive Schrödinger potentials, we know that $SP(\bar \mu,\nu) = \int \bar\varphi d\bar \mu + \int \psi d \nu$ (see e.g. \cite[Theorem 4.7]{Nu22}).
		As  $h(x) \in \partial \Gamma^*(x)$ for $\mu$-a.e. $x$, $(\id,h)_\# \mu$ is concentrated on the sub-gradient of a proper lower semi-continuous function and is therefore optimal for the quadratic cost. Hence, $$\MCov(\mu,\bar \mu) = \int x \cdot h(x) \mu(dx).$$ As $$- \int \log \int e^{y \cdot h(x) + \psi(y)} \nu(dy) \mu(dx) = \int \bar\varphi(h(x)) \mu(dx) = \int \bar\varphi(\bar x)  \bar \mu(d\bar x),$$ we obtain
		\begin{align*}
			\mathcal{D}(\psi) &= \int \psi(y) \nu(dy) + \int \left( h(x) \cdot x - \log \int e^{y \cdot h(x) + \psi(y)} \nu(dy) \right) \mu(dx)\\
			&=\int \psi d\nu + \MCov(\bar \mu ,\mu) + \int \bar\varphi d\bar \mu = SP(\bar \mu,\nu) + \MCov(\bar \mu,\mu).
		\end{align*}
		\vspace{0.5cm}
		
		\noindent\textbf{Proof of (2):} We already know that $m \in \MT(\mu,\nu)$. As 
		\begin{align*}
			\int \bar x\cdot y d\pi(\bar x,y) &= \int h(x) \int y d\pi_{h(x)}(y) \mu(dx) = \int h(x) \int y m_x(dy) \mu(dx)\\ &= \int x \cdot h(x) \mu(dx) = \MCov(\mu,\bar \mu),
		\end{align*}
		we get
		\begin{align*}
			\eqref{eq:introP} &=\eqref{eq:vp} = H(\pi|\bar \mu \otimes \nu) - \int \bar x \cdot y d \pi(\bar x,y) + \MCov(\bar \mu, \mu)\\ 
			&= \int H(\pi_{\bar x}| \nu) \bar \mu (d \bar x)
			= \int H(\pi_{h(x)}|\nu) \mu(dx) = \int H(m_x|\nu) \mu(dx)
			=H(m|\mu \otimes \nu).\qedhere
		\end{align*}
	\end{proof}

	\begin{remark}
		In \cite{NuWi24}, Nutz and Wiesel studied the minimization problem \eqref{eq:introP} in the one-dimensional setting. If $m$ denotes the unique solution of \eqref{eq:introP}, they prove that under some technical assumption (Assumption 2.1 and 2.3), there exists a triplet of functions $(\tilde f, \tilde g, \tilde h) \in L^1(\mu) \times L^1(\nu) \times L^0(m)$ such that $$\log\left(\frac{d m}{d \mu \otimes \nu}\right)(x,y) = \tilde f (x) + \tilde g (y) + \tilde h (x) \cdot (y-x),$$ with $h(x)(y-x) \in L^1(\pi),$ and that such triplet is unique up to an affine shift. This relates to our problem \eqref{eq:main_dual} as follows. If Assumption \ref{ass:exp_affine_unique} is satisfied, let $\psi$ be a solution of the dual problem, let $h(x)$ be the unique solution of \eqref{eq:inner_dual}, and set $\bar \varphi(\bar x) = - \log \int e^{\psi(y)+ \bar x \cdot y} \nu(dy)$. According to Proposition \ref{pro:from_D_to_VP}, we have $$\log \left(\frac{d m_x}{d\nu}\right)(y) = {\bar \varphi(h(x)) + \psi(y) + h(x)\cdot y},$$ from which follows 
		\begin{equation} \label{eq:density_mart_opt}
			\log\left(\frac{dm}{d\mu \otimes \nu}\right)(x,y) = {\bar \varphi(h(x)) + \psi(y) + h(x) \cdot y}.
		\end{equation} This shows that $\big(\bar \varphi(h(x))+ h(x) \cdot x, \psi(y),h(x)\big)$ is a triplet of potentials in the sense of Nutz and Wiesel.
	\end{remark}

	\section{On the continuous-time problems }
	\label{sec:continuous}

	In this section, we focus on the continuous-time problem \eqref{eq:introP_cont} and its connections to the problem \eqref{eq:introP}. We begin with a few definitions:

	\begin{definition}\label{def:class_mart}
		The class $\mathcal A_{\mathrm{mart}}(\mu,\nu)$ of admissible martingales is defined as the set of tuples
		\[
		\left(\Omega,\mathcal F,\mathbb P,(B_t)_{t \in [0,1]},(M_t)_{t \in [0,1]}\right)
		\]
		such that
		\begin{enumerate}
			\item $(\Omega,\mathcal F,\mathbb P)$ is a probability space;
			\item $(B_t)_{t \in [0,1]}$ is a $d$-dimensional standard Brownian motion on $(\Omega,\mathcal F,\mathbb P)$;
			\item  $M_0$ is a $\mathcal F$-measurable random variable independent of $B$ with law $\mu$;
			\item\label{point:condition_volatility} there exists a progressively measurable process $(\sigma_t)_{t \in [0,1]}$ with respect to the filtration generated by $B$ and $M_0$, such that
			\begin{enumerate}
				\item
				\[
				\E\!\left[\int_0^1 \frac{|\sigma_t-I|^2}{1-t}\,dt\right]<+\infty;
				\]
				\item for all $t\in[0,1]$,
				\[
				M_t=M_0+\int_0^t \sigma_s\,dB_s;
				\]
			\end{enumerate}
			\item\label{point:law_mart} $\Law(M_1)= \nu.$
		\end{enumerate}
		For all $M\in\mathcal A_{\mathrm{mart}}(\mu,\nu)$, we define
		\[
		\mathcal C_{\mathrm{mart}}(M) 
		=\frac12 \E\!\left[\int_0^1 \frac{|\sigma_t-I|^2}{1-t}\,dt\right].
		\]
	\end{definition}
	
	Along similar lines, we introduce:
	
	\begin{definition}\label{def:class_drift}
		The class $\mathcal A_{\mathrm{drift}}(\bar{\mu},\nu)$ of martingale-drifted processes is defined as the set of tuples 
		$(\Omega,\mathcal F,\mathbb P,(B_t)_{t \in [0,1]},(X_t)_{t \in [0,1]})$ such that 
		\begin{enumerate}
			\item $(\Omega,\mathcal{F},\mathbb{P})$ is a probability space;
			\item $(B_t)_{t \in [0,1]}$ is a $d$-dimensional standard Brownian motion on $(\Omega,\mathcal{F},\mathbb{P})$;
			\item $X_0$ is a $\mathcal F$-measurable random variable independent of $B$ with law $\bar\mu$;
			\item\label{point:condition_drift} There exists a progressively measurable martingale  $(u_t)_{t\in[0,1]}$ with respect to the filtration generated by $B$ and $X_0$ such that: 
			\begin{enumerate}
				\item $\E\left(\int_0^1 |u_t|^2 dt\right) < + \infty$;
				\item\label{eq:deco_drift}  For all $t \in [0,1]$, $X_t = X_0+ B_t + \int_0^t u_s\,ds$;
			\end{enumerate}
			\item\label{point:law_drift} $\Law(X_1) = \nu.$
		\end{enumerate}
		Given $X \in \mathcal A_{\mathrm{drift}}(\bar\mu,\nu)$, we define 
		\begin{equation*}
			\mathcal C_{\mathrm{drift}}(X)
			=\frac12\,\mathbb E\!\left[\int_0^1 |u_t|^2\,dt\right].
		\end{equation*}
	\end{definition}
	
	We recall that if $Z$ is a random variable defined on a filtered probability space, then its associated Doob martingale is $(\mathbb E[Z|\mathcal F_t])_{t\in[0,1]}$. With a slight abuse of terminology, if $X$ is a stochastic process, then its associated Doob martingale is the Doob martingale associated to its terminal value $X_1$. 
	
	\subsection{A cost-preserving bijection between martingale-drifted processes and admissible martingales}

	Fix $x \in \R^d$ and $m_x \in \PP(\R^d)$.
	
	To lighten the notation, throughout this subsection we write $\mathcal A_{\mathrm{mart}}$ and $\mathcal A_{\mathrm{drift}}$ in place of $\mathcal A_{\mathrm{mart}}(\delta_x,m_x)$ and $\mathcal A_{\mathrm{drift}}(\delta_x,m_x)$, respectively. In \cite[Section 2]{ElMi20}, the authors observed that any martingale can be associated with a \emph{martingale drift} via an application of the stochastic Fubini theorem. Moreover, they show that the Föllmer drift can be expressed in terms of the volatility of the corresponding Doob martingale. In fact, this construction does not rely on the specific structure of the Föllmer drift, but applies more generally to any martingale drift. We formalize and extend this observation by proving that these operations define a cost-preserving correspondence between the set of martingale drifts and the set of admissible martingales (from $\delta_x$ to $m_x$).

	\begin{theorem}\label{them:fondamental_bijection}
		The map
		\[
		\Phi:\mathcal A_{\mathrm{drift}}\longrightarrow \mathcal A_{\mathrm{mart}}
		\]
		defined by
		\[
		\Phi(X)_t = \mathbb E\big[X_1\mid \sigma(\{B_s, s \in [0,t]\})  \big],
		\qquad t\in[0,1],
		\]
		is a bijection between $\mathcal A_{\mathrm{drift}}$ and $\mathcal A_{\mathrm{mart}}$.
		Its inverse $\Phi^{-1}:\mathcal A_{\mathrm{mart}}\to \mathcal A_{\mathrm{drift}}$
		is given by
		\begin{equation*}
			\Phi^{-1}(M)_t = x+B_t + \int_0^t  u_s ds,
		\end{equation*}
		where, if the dynamic of $M$ is given by $M_t = \sigma_t d B_t$, then
		$ u_s := \int_0^s \frac{\sigma_r - I}{1-r}dB_r$. 
		Moreover, the costs are preserved: for all $(X, M) \in  \mathcal A_{\mathrm{drift}} \times \mathcal A_{\mathrm{mart}}$, it holds
		\[
		\mathcal C_{\mathrm{mart}}\big(\Phi(X)\big)=\mathcal C_{\mathrm{drift}}(X),
		\qquad
		\mathcal C_{\mathrm{drift}}\big(\Phi^{-1}(M)\big)
		=\mathcal C_{\mathrm{mart}}(M).
		\]
	\end{theorem}
	
	\begin{remark}
		We write $\Phi(X) = M$ to make the notation lighter, but to be fully exhaustive, one should rather write $\Phi((\Omega,\mathcal{F},\mathbb{P}, B, X^x)) = (\Omega,\mathcal{F},\mathbb{P}, B, M^x)$ (and similarly for $\Phi^{-1}$).
	\end{remark}	
	
	\begin{proof}
		\textbf{From drifted to martingale:} Let $X\in\mathcal A_{\mathrm{drift}}$ and $u$ be as in Definition \ref{def:class_drift}, and write $M = \Phi(X)$, i.e., if $(\mathcal{F}_t)_{t \in [0,1]}$ stands for the Brownian filtration, $M_t = \E[X_1|\mathcal{F}_t]$ for all $t \in [0,1]$.	Since $u$ is a continuous square--integrable martingale with respect to the Brownian filtration, there exists a progressively measurable process $\alpha=(\alpha_t)_{t\in[0,1]}$ with respect to the Brownian filtration, by the martingale representation theorem (see e.g. \cite[Theorem 4.33]{Okse03}),  such that
		$
		\E\!\int_0^1 |\alpha_t|^2\,dt < \infty
		$ 
		and   
		\begin{equation}\label{eq:u-rep}
			u_t = \int_0^t \alpha_s\,dB_s,
			\qquad t\in[0,1].
		\end{equation}
		By stochastic Fubini, we have
		\[
		\int_0^1 u_s\,ds
		=
		\int_0^1\!\left(\int_0^s \alpha_r\,dB_r\right)\!ds
		=
		\int_0^1 (1-r)\alpha_r\,dB_r.
		\]
		Hence
		\begin{equation}\label{eq:X1-as-int}
			X_1
			=
			x+ B_1+\int_0^1 u_s\,ds
			=
			x+ \int_0^1 \Big(I+(1-t)\alpha_t\Big)\,dB_t.
		\end{equation}
		Define $\sigma_t = I+(1-t)\alpha_t$ for all $t \in [0,1]$.
		Taking conditional expectation in \eqref{eq:X1-as-int} yields
		\[
		M_t=\E[X_1\mid \mathcal{F}_t]
		=
		x+ \int_0^t \sigma_s\,dB_s,\qquad t\in[0,1].
		\]
		As $
		\sigma_t-I = (1-t)\alpha_t$, we have
		\begin{equation}\label{eq:alpha_and_sigma}
			\frac{|\sigma_t-I_d|^2}{1-t}
			=
			(1-t)\,|\alpha_t|^2.
		\end{equation}
		Therefore,
		\begin{equation*}
			\E\!\int_0^1 \frac{|\sigma_t-I|^2}{1-t}\,dt
			=
			\E\!\int_0^1 (1-t)\,|\alpha_t|^2\,dt
			<\infty,
		\end{equation*}
		which shows $\sigma$ satisfies Condition \eqref{point:condition_volatility} of Definition \ref{def:class_mart}. As $M_1 = \E[X_1|\mathcal{F}_1] = X_1$, $M$ also satisfies Condition \eqref{point:law_mart} of Definition \ref{def:class_mart}, which proves $M \in \mathcal A_{\textrm{mart}}$. Moreover, 
		\begin{align*}
			(\Phi^{-1} \circ \Phi)(X)_t &= \Phi^{-1}(M)_t = x+ B_t + \int_0^t \left(\int_0^s \frac{\sigma_r-I}{1-r} d B_r \right) ds \\
			&= x+ B_t + \int_0^t \left(\int_0^s \alpha_r d B_r \right) ds =B_t + \int_0^t u_s ds = X_t.
		\end{align*}
		Finally, by It\^o's isometry applied to \eqref{eq:u-rep}, we have 
		$
		\E|u_t|^2
		=
		\E\!\int_0^t |\alpha_s|^2\,ds
		$. 
		Integrating over $t\in[0,1]$ and using Fubini yields,
		\[
		\E\!\int_0^1 |u_t|^2\,dt
		=
		\E\!\int_0^1 (1-s)\,|\alpha_s|^2\,ds.
		\]
		Combined with Equation \eqref{eq:alpha_and_sigma}, this shows
		\[
		\frac12\,\E\!\int_0^1 |u_t|^2\,dt
		=
		\frac12\,\E\!\int_0^1 \frac{|\sigma_t-I|^2}{1-t}\,dt,
		\]
		that is $\mathcal C_{\mathrm{drift}}(X)
		=
		\mathcal C_{\mathrm{mart}}(\Phi(X)).$\\
		\newline
		\textbf{From martingale to drifted:} Let $M\in\mathcal A_{\mathrm{mart}}$, $\sigma$ be as in Definition \ref{def:class_mart}, and set $X = \Phi^{-1}(M)$, i.e.,
		\begin{equation*}
			X_t = x+ B_t + \int_0^t u_s\,ds,\qquad \text{where} \quad u_t =\int_0^t \frac{\sigma_s-I}{1-s}\,dB_s,\qquad t\in[0,1[. 
		\end{equation*}	
		By Fubini,
		\begin{equation*}\label{eq:Fubini_drift}
			\int_0^1 \E\!\int_0^t \frac{|\sigma_s-I|^2}{(1-s)^2}\,ds\,dt\\
			=
			\E\!\int_0^1 \frac{|\sigma_s-I|^2}{(1-s)^2}\,(1-s)\,ds
			=
			\E\!\int_0^1 \frac{|\sigma_s-I|^2}{1-s}\,ds < + \infty,
		\end{equation*}
		which ensures $(u_t)_{t \in [0,1]}$ is a well defined martingale with respect to the Brownian filtration. By It\^o's isometry and Fubini,
		\begin{equation}\label{eq:finite_drift}
			\E\!\int_0^1 |u_t|^2\,dt
			=
			\int_0^1 \E|u_t|^2\,dt = \int_0^1 \int_0^t \frac{|\sigma_s-I|^2}{(1-s)^2} ds dt 
			= \E\!\int_0^1 \frac{|\sigma_s-I|^2}{1-s}\,ds,
		\end{equation}
		so $u$ satisfies Condition \eqref{point:condition_drift} of Definition \ref{def:class_drift}. By stochastic Fubini,
		\[
		\int_0^1 u_s\,ds
		=
		\int_0^1 \left(\int_0^s \frac{\sigma_r-I}{1-r}\,dB_r\right) ds
		=
		\int_0^1 (1-r)\,\frac{\sigma_r-I}{1-r}\,dB_r
		=
		\int_0^1 (\sigma_r-I)\,dB_r.
		\]
		Therefore,
		\[
		X_1
		=
		x+B_1+\int_0^1 u_s\,ds
		=
		x+\int_0^1 I\,dB_s + \int_0^1 (\sigma_s-I)\,dB_s
		= x+
		\int_0^1 \sigma_s\,dB_s
		=
		M_1.
		\]
		In particular, $\Law(X_1)=\Law(M_1)=m_x$, so $X$ satisfies Condition \eqref{point:law_drift} of Definition \ref{def:class_drift}, which concludes $X\in\mathcal A_{\mathrm{drift}}$. Moreover, as $M$ is a martingale with respect to the Brownian filtration $(\mathcal{F}_t)_{t \in [0,1]}$,  
		$$
		\Phi \circ \Phi^{-1}(M)_t =\Phi(X)_t
		=
		\E[X_1\mid\mathcal F_t]
		=
		\E[M_1\mid\mathcal F_t]
		=
		M_t \quad \text{for all } t \in [0,1].$$ The equality
		$\mathcal\mathcal C_{\mathrm{mart}}(M) =  C_{\mathrm{drift}}(\Phi^{-1}(M))$ follows directly by Equation \eqref{eq:finite_drift}.
	\end{proof}
	
	\begin{remark}\label{remq:drift_to_martingale_drift}
		Assume that $(\Omega, \mathcal F, \mathbb P, B, X^x)$ satisfies all the requirements of Definition \ref{def:class_drift}, except that we do not ask for $(u_t)_{t \in [0,1]}$ to be a martingale. Denoting by $M^x$ the Doob martingale associated with $X^x$, one can show that $M^x$ belongs to $\mathcal A_{\mathrm{mart}}(\delta_x,m_x)$ and therefore canonically induces a drifted process $\tilde X^x \in \mathcal A_{\mathrm{drift}}(\delta_x,m_x)$ via $\Phi$.	Remarkably, this drifted process $\tilde X^x$ -- characterized as the unique process with martingale drift sharing the same terminal value as $X^x$ -- has lower energy than the original drift. Moreover, the two energies coincide if and only if the corresponding processes are almost surely equal. The technical proof, relying on the Hilbert space structure of $L^2(\Omega \times [0,1])$, can be found in \cite[Appendix B]{BaBeBiLe26v1}.
	\end{remark}
	
	The correspondence between the classes $\mathcal{A}_{\mathrm{drift}}$ and $\mathcal{A}_{\mathrm{mart}}$ will prove itself useful, as it will allow us to apply standard result on  drifts to solve \eqref{eq:introP_cont}. We will make use of the following lemma, which follows for instance from  \cite[Proposition 1]{Le13}:
	
	\begin{lemma}\label{lem:energy_entropy_bound}
		Let $\left(\Omega,\mathcal F,\mathbb P,(B_t)_{t \in [0,1]},(X_t^x)_{t \in [0,1]}\right)\in \mathcal{A}_{\mathrm{drift}}$. Then, $H(\Law_\mathbb{P}(X^x)\mid \mathbb{W}_x) \le C_{\mathrm{drift}}(X^x).$
	\end{lemma}

	\subsection{On the Föllmer process of Schrödinger bridges}
	Define
	$\Omega = \mathcal{C}([0,1],\R^d)$, $\mathcal{F} = \mathcal{B}(\Omega)$, and let $X = (X_t)_{t \geq 0}$ be the canonical process, defined by $X_t(\omega) = \omega(t)$ for all $(t,\omega) \in [0,1] \times \Omega$.
	
	\subsubsection{Schrödinger bridge problem}
	Fix $x \in \R^d$ and let $m_x$ be a probability measure with barycentre $x$. Denote by $\mathbb{W}^x$ the law of a standard Brownian motion started at $x \in \R^d$. The continuous-time Schrödinger (Bridge) problem is, in this particular case, defined as  
	\begin{equation}\label{eq:SB}
		\inf_{\mathbb{Q} \in \PP(\Omega),{X_0}_\# \mathbb{Q} = \delta_x, {X_1}_\# \mathbb{Q} = m_x} H(\mathbb{Q}|\mathbb{W}^x).\tag{SB}
	\end{equation}
	
	The following is well-known and direct to prove:

	\begin{proposition}
		Assume $H(m_x|\gamma_x)<+ \infty$. Then the problem \eqref{eq:SB} admits a unique solution, called Schrödinger bridge from $x$ to $m_x$ and denoted $\mathbb{Q}^x$. Moreover,  \begin{equation*}
			\label{eq:rho_x}\frac{d\mathbb{Q}^x}{d \mathbb{W}^x}(\omega) = \rho_x(\omega_1), \qquad \text{where } \rho_x := \frac{d m_x}{d \gamma_x}, 
		\end{equation*}
		and it holds $H(\mathbb{Q}^x|\mathbb{W}^x) = H(m_x|\gamma_x).$
	\end{proposition}
	
	In fact the term 'Schrödinger bridge' is more typically employed for the solution of \eqref{eq:SB} when the initial distribution is not a Dirac, while in the Dirac case the terminology 'h-transform' or 'conditioned Brownian motion' is often preferred in the classical literature. For simplicity, we only use the former terminology.

	\subsubsection{Föllmer drift of the Schrödinger bridge}
	Given a probability measure $\mathbb{Q}$ absolutely continuous with respect to the standard Wiener measure $\mathbb{W}$ on $\R_+$, Föllmer proved the existence of a progressively measurable process $(u_t)_{t \geq 0}$ on $\mathcal C(\R_+,\R^d)$ with values in $\R^d$ such that:
	\begin{enumerate}
		\item For all $t \geq 0$ and $\omega \in \mathcal C(\R_+,\R^d)$, $\int_0^t |u_s(\omega)|^2 ds < +\infty$;
		\item $\Law_{\mathbb{Q}}(X - U) = \mathbb{W}$ where $U$ is defined by $U_t = \int_0^t u_s\,ds$ for all $t \geq 0$;
		\item $H(\mathbb{Q}\mid\mathbb{W}) = \frac{1}{2} \mathbb E_{\mathbb{Q}}\int_0^{+\infty} |u_t(\omega)|^2dt $.
	\end{enumerate}
	The process $(u_t)_{t \geq 0}$ is called the \emph{Föllmer drift} associated with $\mathbb{Q}$, while $(\Omega,\mathcal{F},\mathbb{Q},X)$ is referred to as the \emph{Föllmer process} of $\mathbb{Q}$ (or associated to $\mathbb{Q}$).
	
	In the following, we compute the Föllmer drift when $\mathbb{Q}$ is the Schrödinger bridge from $\delta_x$ to $m_x$, as well as the volatility of the associated Doob martingale. The expression of the Föllmer drift is classical in the literature on h-transforms, but see Lehec's \cite[Theorem~12]{Le13} for a recent reference.  An analogue computation of the volatility of the Doob martingale associated to the Föllmer drift of the Schrödinger bridge can be found in Eldan and Mikulincer's \cite{ElMi20}. 	See \cite[Appendix C]{BaBeBiLe26v1} for a self-contained proof.
	\begin{proposition}\label{prop:expression_Follmer_drift}
		Let $x \in \R^d$, $m_x \in \PP(\R^d)$ absolutely continuous with respect to $\gamma_x$, and define $\rho_x = \frac{dm_x}{d \gamma_x}.$ As in the previous subsection $\mathbb Q_x$ denotes the Schrödinger bridge from $\delta_x$ to $m_x.$ Let $(B_t^x)_{t \in [0,1]}$ be the process defined by 
		\begin{equation*}\label{eq:time_changed_BM}
			B_t^x  = X_t - x - \int_0^t u_s^x(X_s)ds,
		\end{equation*}
		where $$u_s^x(z) := \nabla \log (\gamma^{1-s}*\rho_x)(z).$$ Then:
		\begin{enumerate}
			\item\label{point:expression_drift} The tuple $X^x := (\Omega, \mathcal{F}, \mathbb{Q}^x, (B_t^x)_{t \in [0,1]}, X)$ belongs to $\mathcal A_{\mathrm{drift}}$ and
			\begin{equation*}\label{eq:equality_drift}
				\mathcal{C}_{\textrm{drift}}(X^x) := \frac{1}{2}\E\left(\int_0^1 |u_t^x|^2 dt\right) = H(\mathbb{Q}^x|\mathbb{W}^x) = H(m_x|\gamma_x).
			\end{equation*}
			This shows that $(u_t^x)_{t \geq 0}$ is the Föllmer drift of $\mathbb Q_x.$
			\item\label{point:law_Follmer_drift} It holds
			$$\Law_{\mathbb{Q}^x}(X_1|X_t=z) =  \frac{ \rho_x(y) \phi_{1-t}(y-z)}{\int \rho_x(\tilde y) \phi_{1-t}(\tilde y-z)d \tilde y }dy.$$
			\item\label{point:expression_martingale} The Doob martingale $M^x =(M_t^x)_{t \in [0,1]}$ canonically associated to $X^x$, i.e., 
			$$M_t^x = \E_{\mathbb{Q}^x}[X_1|\mathcal{F}_t^x], \qquad \mathcal{F}_t^x = \sigma(\{B_s^x~;~0 \leq s \leq t\}),$$
			satisfies
			\begin{equation}\label{eq:explicit_M_via_X}
				M_t^x = X_t^x + (1-t) u_t^x(X_t^x) \quad \text{and} \quad dM_t^x = \sigma_t^x(M_t^x) d B_t^x,
			\end{equation}
			where $\sigma_t^x := I + (1-t) \nabla^2 \log (\gamma^{1-t}*\rho_x).$ Moreover,
			\begin{equation*}
				H(m_x|\gamma_x) = \frac12 \E \left( \int_0^1 \frac{|\sigma_t^x - I|^2}{1-t}dt\right).
			\end{equation*}
		\end{enumerate}
	\end{proposition}
	
	\begin{remark}\label{remq:uniquess_of_Pcont}
		Given $M^x \in \mathcal A_{\mathrm{mart}}(\delta_x,m_x)$, Lemma \ref{lem:energy_entropy_bound}, Theorem \ref{them:fondamental_bijection} and the data processing inequality yield $$\mathcal{C}_{\mathrm{mart}}(M^x) = \mathcal{C}_{\mathrm{drift}}(\Phi^{-1}(M^x)) \geq H(\Law(\Phi^{-1}(M^x))|\mathbb{W}_x) \geq H(m_x|\gamma_x).$$ Moreover, the previous inequality is an equality if and only if $\Law(\Phi^{-1}(M^x))$ is the Schrödinger bridge from $\delta_x$ to $m_x$ and $\Phi^{-1}(M^x)$ is the Föllmer process of this Schrödinger bridge. Therefore $\mathcal{C}_{\mathrm{mart}}(M^x) = H(m_x|\gamma_x)$ if and only if $M^x$ is the Doob martingale associated to the Föllmer drift of the Schrödinger bridge from $\delta_x$ to $m_x$. In particular $\Law(M^x)$ is completely determined in this case. This facts were already established in \cite{ElMi20} (see also \cite{CoFaMi24})
	\end{remark}
	The Föllmer drift and the Föllmer volatility can both be expressed via the backward transition kernel  $q_{1-t}^{x,z} := \Law_{\mathbb{Q}^x}(X_1|X_t = z)$. By Proposition \ref{prop:expression_Follmer_drift} it holds 
	\begin{equation}\label{eq:def-q-main}
		q_{1-t}^{x,z}(dy)
		=\frac{\phi_{1-t}(z-y)\,\rho_x(y)}{\int_{\R^d}\phi_{1-t}(z-\tilde y)\,\rho_x(\tilde y)\,d\tilde y}dy.
	\end{equation}
	and a straightforward computation yields
	\begin{equation}\label{eq:grad-logH-moment}
		u_t^x(z)
		=\frac{1}{1-t}\Big(\E_{q_{1-t}^{x,z}}[Y]-z\Big) \quad \text{and} \quad \sigma_t^x(z) =\frac{1}{1-t}Cov_{q_{1-t}^{x,z}}(Y).
	\end{equation}
	These formulae appear in \cite{ElMi20} too, or in \cite[Appendix A]{BaBeBiLe26v1} with a detailed proof.
	\subsection{Equivalence of \eqref{eq:introP} and \eqref{eq:introP_cont}}
	
	We now rely on the above result on the Föllmer drift of Schrödinger bridges to provide an explicit correspondence between the optimizer of \eqref{eq:introP} and \eqref{eq:introP_cont}. Throughout the whole subsection, we will assume that
	\begin{equation}\label{eq:reinforced_assumption_gaussian}
		H(\nu|\gamma)< + \infty. \tag{H}
	\end{equation}
	\begin{remark}\label{rem:implications_assumptions}
		This is a reinforcement of Assumption \eqref{ass:int_nu_hyperplane}. The second part of the latter follows directly from the fact that $\gamma$ does not charge any hyperplane. For the first part, recall that for any measurable function $\Phi \ge 0$ such that 
		$\int e^{\Phi} \, d\gamma < +\infty$, we have
		\[
		\int \Phi \, d\nu \le H(\nu|\gamma) + \log \int e^{\Phi} \, d\gamma.
		\]
		Taking $\Phi(y) = \alpha |y|^2$ for any $\alpha \in (0,1/2)$, we obtain
		$
		\int e^{\alpha |y|^2} \, \nu(dy) < +\infty
		$.  
		Using Young's inequality $ab \le a^2/2 + b^2/2$ with 
		$a = q/\sqrt{2\alpha}$ and $b = \sqrt{2\alpha}|y|$, we get
		$
		q|y| \le \frac{q^2}{4\alpha} + \alpha |y|^2
		$ 
		and therefore $\int e^{q|y|} \, \nu(dy)
		\le e^{q^2/(4\alpha)} \int e^{\alpha |y|^2} \, \nu(dy)
		< +\infty.$
	\end{remark}
	\subsubsection{Gaussian analogues}
	
	For the continuous time problem, it is more natural to consider the \emph{Gaussian} analogue of \eqref{eq:introP}
	\begin{align}\label{eq:introP_G}\tag{$P^{G}$}
		\inf_{m\in\MT(\mu,\nu)} H(m|\mu.\gamma).
	\end{align}
	Indeed, this problem rewrites as $\inf_{m\in\MT(\mu,\nu)} \int H(m_x|\gamma_x)\mu(dx)$, 
	and Remark \ref{remq:uniquess_of_Pcont} shows that $H(m_x|\gamma_x)$ is related to \eqref{eq:introP_cont} with deterministic start, via the equality
	\begin{equation*}
		H(m_x|\gamma_x) = \inf_{M^x \in \mathcal A_{\mathrm{mart}}(\delta_x,m_x)} C_{\mathrm{mart}}(M^x),
	\end{equation*}
	which we know to be achieved by the Doob martingale of the Föllmer process associated to the Schrödinger bridge from $\delta_x$ to $m_x$. The following lemma proves that the problems \eqref{eq:introP} and \eqref{eq:introP_G} have the same optimizer.
	
	\begin{lemma}\label{lem:relation_gaussian_analogue}
		Assume Assumption \eqref{eq:reinforced_assumption_gaussian} is satisfied.
		For all $m \in \MT(\mu,\nu)$, it holds 
		\begin{equation}\label{eq:link_entropies}
			H(m|\mu \otimes \nu)+ H(\nu|\gamma) = H(m |\mu \cdot \gamma) + \frac{m_2(\mu)}{2}.
		\end{equation}
		In particular, 
		$\Argmin_{m \in \MT(\mu,\nu)} H(m|\mu \otimes \nu) = \Argmin_{m \in \MT(\mu,\nu)} H(m|\mu \cdot \gamma)$ and $$value\eqref{eq:introP} + H(\nu|\gamma) = value\eqref{eq:introP_G} + \frac{m_2(\mu)}{2}.$$
	\end{lemma}
	
	\begin{proof}
		Let $m \in \MT(\mu,\nu)$. As $\nu \ll \gamma$, it holds $\mu \otimes \nu \ll \mu \otimes \gamma$ and $\frac{d \mu \otimes \nu}{d \mu \otimes \gamma}(x,y) = \frac{d\nu}{d\gamma}(y)$. Thus, if $m \ll \mu \otimes \nu$, then $\frac{dm}{d \mu \otimes \gamma}(x,y) = \frac{dm}{d \mu \otimes \nu}(x,y) \frac{d\nu}{d\gamma}(y)$. Applying the logarithm and integrating against $m$ yields 
		\begin{equation}\label{eq:change_ref_prod}
			H(m|\mu \otimes \gamma) = H(m|\mu \otimes \nu) + H(\nu|\gamma).  
		\end{equation}
		If $m \not\ll \mu \otimes \nu$, as $m \in \MT(\mu,\nu)$, $m \not\ll \mu \otimes \gamma$, so that the previous equality still holds. Consequently, it holds in all cases. Now, note that $m \ll \mu \otimes \gamma$ if and only if $m \ll \mu \cdot \gamma$. In that case,
		\begin{align*}
			H(m|\mu \otimes \gamma) &= \int \log \left( \frac{dm}{d \mu \cdot \gamma} \right) + \log \left( \frac{d \mu \cdot \gamma}{\mu \otimes \gamma} \right) dm \\ &= H(m|\mu \cdot \gamma) + \int x\cdot y - \frac {x^2} 2 dm(x,y) = H(m |\mu \cdot \gamma) + \frac{m_2(\mu)}{2}.
		\end{align*}
		As this equality also holds when $m \not \ll \mu \otimes \gamma$, together with Equation \eqref{eq:change_ref_prod}, this concludes the proof of Equation \eqref{eq:link_entropies}. The last part of the statement directly follows.
	\end{proof}
	
	\subsubsection{The randomization procedure}
	We now present a conditioning and a randomization procedure for admissible martingales, which will allow us carry our optimization results from $\mathcal A_{\mathrm{mart}}(\delta_x,m_x)$ to $\mathcal A_{\mathrm{mart}}(\mu,\nu).$\\
	
	\textbf{Conditioning.}  
	Any admissible martingale
	$
	\left(\Omega,\mathcal F,\mathbb P,(B_t)_{t \in [0,1]},(M_t)_{t \in [0,1]}\right) \in \mathcal A_{\mathrm{mart}}(\mu,\nu)
	$ 
	induces a measurable family of admissible martingales with deterministic initial condition. More precisely, for $\mu$-a.e.\ $x$, set $m_x = \Law(M_1|M_0 = x)$ and $\mathbb P_x=\mathbb P(\,\cdot\,|M_0=x)$.
	This defines an admissible martingale
	\[
	\left(\Omega,\mathcal F,\mathbb P_x,(B_t)_{t \in [0,1]},(M_t)_{t \in [0,1]}\right)\in \mathcal A_{\mathrm{mart}}(\delta_x,m_x),
	\]
	with dynamics
	$
	M_t = x + \int_0^t \sigma_s\,dB_s
	\, (\mathbb P_x\text{-a.s.})
	$. 
	Moreover, the cost is compatible with this disintegration in the sense that it satisfies
	$
	\mathcal C_{\mathrm{mart}}(M)
	=
	\int_{\R^d} \mathcal C_{\mathrm{mart}}(M^x)\,\mu(dx)
	$.
	
	\textbf{Randomization.}  
	Conversely, assume we have admissible martingales
	\[
	\left(\Omega^x,\mathcal F^x,\mathbb P^x,(B_t^x),(M_t^x)\right)
	\in \mathcal A_{\mathrm{mart}}(\delta_x,m_x), \quad x \in \R^d,
	\]
	such that $\nu = \int m_x \mu(dx)$, with dynamics
	$
	M_t^x = x + \int_0^t \sigma_s^x(B^x)\,dB_s^x
	$. 
	Then one can represent all these martingales on a common canonical space by considering, for each $x$, the law
	\[
	\Gamma^x:=\Law_{\mathbb P^x}(B^x,M^x)
	\]
	on $C([0,1],\R^d)\times C([0,1],\R^d)$, and then randomize them by defining the product measure
	\[
	\bar{\mathbb P}(dx,db,dm):=\mu(dx)\,\Gamma^x(db,dm)
	\]
	on the space $ \Omega := \R^d \times C([0,1],\R^d)\times C([0,1],\R^d)$. The processes $B$ and $M$ defined by
	\[
	B_t(x,b,m)=b_t,\qquad M_t(x,b,m)=m_t,
	\]
	satisfy
	$	M_t = M_0 + \int_0^t \sigma_s\,dB_s
	\, ( \bar{\mathbb P}\text{-a.s.})
	$,	
	where the volatility is given by
	$\sigma_s(x,b,m) = \sigma_s^{x}(b)$. It is straightforward to verify that $M \in \mathcal A_{\mathrm{mart}}(\mu,\nu)$. Moreover, its cost is obtained by aggregating the fiberwise costs:
	$
	\mathcal C_{\mathrm{mart}}(M)
	=
	\int_{\R^d} \mathcal C_{\mathrm{mart}}(M^x)\,\mu(dx)
	$. 
	From now on we will refer to $M$ as the $\mu$-randomization of $(M^x)_{x \in \R^d}.$

	\subsubsection{Correspondence between the optimizers of \eqref{eq:introP_G} and \eqref{eq:introP_cont} }

	The following theorem states that \eqref{eq:introP_G} and \eqref{eq:introP_cont} have the same values and provides an explicit correspondence between their minimizers.
	
	\begin{theorem}
		Assume Assumption \eqref{eq:reinforced_assumption_gaussian} is satisfied.
		\begin{enumerate}
			\item The values of the problems \eqref{eq:introP_G} and \eqref{eq:introP_cont} are equal.
			\item Assume \eqref{eq:introP_G} has finite value and is attained by $m(dx,dy) = \mu(dx)m_x(dy) \in \MT(\mu,\nu)$. For all $x \in \R^d$, let $M^x$ be the Doob martingale of the Föllmer process associated to the Schrödinger bridge from $\delta_x$ to $m_x$, and let $M$ be the $\mu$-randomization of $(M^x)_{x \in \R^d}.$ Then \eqref{eq:introP_cont} is attained by $M$.
			\item If \eqref{eq:introP_cont} has finite value and is attained by $M$, then  \eqref{eq:introP}  is attained by $\Law(M_0,M_1)$. Moreover, the solution of \eqref{eq:introP_cont} is unique in law.
		\end{enumerate}
	\end{theorem}

	\begin{proof}
		\begin{enumerate}
			\item Consider $M \in \mathcal A_{\mathrm{mart}} (\mu,\nu)$ and let  $m = \mu(dx) m_x(dy) \in \MT(\mu,\nu)$ denote the law of $(M_0,M_1)$. In the conditioning procedure of the previous subsection, we constructed admissible martingales $M^x \in \mathcal A (\delta_x, m_x)$, $x \in \R^d$ such that $\mathcal C_{\mathrm{mart}}(M)
			=
			\int_{\R^d} \mathcal C_{\mathrm{mart}}(M^x)\,\mu(dx)$. 
			By Remark \ref{remq:uniquess_of_Pcont}, we know that $C_{\mathrm{mart}}(M^x) \ge H(m_x|\gamma_x)$, so that 
			\begin{equation*}\label{eq:Gaussian_bound}
				\mathcal C_{\mathrm{mart}}(M) \geq \int H(m_x|\gamma_x) \mu(dx) \geq \eqref{eq:introP_G}.
			\end{equation*}
			By taking the infimum over all $M \in \mathcal{A}_{\mathrm{mart}}(\mu,\nu)$, we obtain $\eqref{eq:introP_cont} \geq \eqref{eq:introP_G}$. To prove the converse inequality, let $m(dx,dy) = \mu(dx) m_x(dy) \in \MT(\mu,\nu).$ For all $x \in \R^d$, let $M^x$ be the Doob martingale of the Föllmer process associated to the Schrödinger bridge from $\delta_x$ to $m_x$. In Point \ref{point:expression_martingale} of Theorem \ref{them:fondamental_bijection}, we proved that $M^x \in \mathcal A (\delta_x, m_x)$ and $C_{\mathrm{mart}}(M^x) = H(m_x|\gamma_x).$  The $\mu$-randomization $M$ of $(M^x)_{x \in \R^d}$ satisfies $M \in \mathcal{A}_{\mathrm{mart}}(\mu,\nu)$ and $C_{\mathrm{mart}}(M) = \int C_{\mathrm{mart}}(M^x) \mu(dx),$ so
			\begin{equation}
				\eqref{eq:introP_cont} \leq C_{\mathrm{mart}}(M) = \int C_{\mathrm{mart}}(M^x) \mu(dx) = \int H(m_x|\gamma_x) \mu(dx).
			\end{equation}
			Taking the infimum over $m \in \MT(\mu,\nu)$ yields $\eqref{eq:introP_cont} \leq \eqref{eq:introP_G}$, which concludes.
			\item As seen in the proof of the previous point $ C_{\mathrm{mart}}(M) = \int H(m_x|\gamma_x)\mu(dx).$ As $\int H(m_x|\gamma_x)\mu(dx) = \eqref{eq:introP_G} = \eqref{eq:introP_cont}$, this concludes.
			\item Let $M$ be an optimizer of \eqref{eq:introP_cont} and define $m = \Law(M_0,M_1)$. By Equation \eqref{eq:Gaussian_bound}, it holds
			\begin{equation}\label{eq:sandwich_Gaussian_bound}
				\eqref{eq:introP_G} =\eqref{eq:introP_cont} =  C_{\mathrm{mart}}(M) \geq \int H(m_x|\gamma_x) \mu(dx) \geq \eqref{eq:introP_G}.
			\end{equation}
			Thus $\int H(m_x|\gamma_x) \mu(dx) = \eqref{eq:introP_G}$, which proves the first part of the statement. Regarding uniqueness, by Equation \eqref{eq:sandwich_Gaussian_bound},  $\int C_{\mathrm{mart}}(M^x) \mu(dx) = C_{\mathrm{mart}}(M) =  \int H(m_x|\gamma_x) \mu(dx).$ As $C_{\mathrm{mart}}(M^x) \geq H(m_x|\gamma_x)$, this yields $C_{\mathrm{mart}}(M^x) = H(m_x|\gamma_x)$ for $\mu$-a.e. $x \in \R^d.$ By Remark \ref{remq:uniquess_of_Pcont}, for any such $x$,  $\Law(M^x)$ is characterized as the law of the Doob martingale of the Föllmer process associated to the  Schrödinger bridge from $\delta_x$ to $m_x$. As \eqref{eq:introP} is uniquely attained and $\Law(M) = \int \Law(M^x) \mu(dx)$, this concludes.  \qedhere 
		\end{enumerate}
	\end{proof}

	\section{Behaviour under scaling and connection to filtering}
	\label{sec:filtering}
	
	We extend the discussion given in Part \ref{sec:intro_filter} of the Introduction.  
	Here \(m\in\MT(\mu,\nu)\), we write \(m(dx,dy)=\mu(dx)m_x(dy)\), and let \(\sigma>0\).
	For each \(x\), consider the F\"ollmer process from \(\delta_x\) to \(m_x\) relative to the reference process \(\sigma B\), and denote by \(M^\sigma=(M_t^\sigma)_{t\in[0,1]}\) the corresponding Doob martingale.
	Thus \(M^\sigma\) is the F\"ollmer martingale associated with the coupling \(m\) and the reference volatility \(\sigma\).

	Let \((X,Y)\sim m\), let \(W\) be an independent Brownian motion in \(\R^d\).
	Here \(Y-X\) represents a hidden (static) signal, which one tries to estimate from the observation process
	\[
	R_s:=s(Y-X)+W_s, \qquad s\ge 0.
	\]
	Considering the optimal estimator 
	$
	Z_s:=\E[Y\mid X,(R_r)_{0\le r\le s}]$ for $ s\ge 0$, observe that $Z_0=\mathbb E[Y|X]=X$, that process $Z$ is a martingale, and $Z_{s}\to Y$ since $R_s/s\to Y-X$ as $s\to\infty$.
	We have:

	\begin{theorem}[Scaling and filtering representation]\label{thm:intro_scaling_filter}
		Consider the deterministic time change 
		\( 
		\tau_\sigma(s):=\frac{\sigma^2 s}{1+\sigma^2 s},\) \( s\ge 0.
		\)
		Then \(Z=(Z_s)_{s\ge 0}\) is a continuous Markov martingale, and for every \(\sigma>0\),
		\[
		(M^\sigma_{\tau_\sigma(s)})_{s\ge 0}\stackrel{d}{=}(Z_s)_{s\ge 0}.
		\]
		In particular, the law of the time-changed process \((M^\sigma_{\tau_\sigma(s)})_{s\ge 0}\) does not depend on \(\sigma\).
	\end{theorem}
	
	\begin{proof}
		Since both constructions are fiberwise in the disintegration \(m(dx,dy)=\mu(dx)m_x(dy)\), it suffices to prove the statement conditionally on \(X=x\), that is, for the case \(\mu=\delta_x\) and terminal law \(m_x\). The general result then follows by integrating over \(x\sim\mu\).
		
		Fix \(x\in\R^d\), let \(Y\sim m_x\), and let \(\widetilde X^\sigma\) be the F\"ollmer process from \(\delta_x\) to \(m_x\) relative to the reference process \(\sigma B\). Its associated Doob martingale is \(M_t^\sigma=\E[Y\mid (\widetilde X_u^\sigma)_{0\le u\le t}]\). Conditionally on \(Y=y\), the process \(\widetilde X^\sigma\) is a Brownian bridge from \(x\) to \(y\) with volatility \(\sigma\). Hence, on a suitable probability space, one may write
		\[
		\widetilde X_t^\sigma=x+t(Y-x)+\sigma \beta_t,\qquad 0\le t<1,
		\]
		where \(\beta\) is a standard Brownian bridge independent of \(Y\). Using the classical representation \(\beta_t=(1-t)\widetilde W_{t/(1-t)}\) for a Brownian motion \(\widetilde W\) in \(\R^d\), this becomes
		\[
		\widetilde X_{\tau_\sigma(s)}^\sigma
		=
		x+\tau_\sigma(s)(Y-x)+\frac{\sigma}{1+\sigma^2 s}\widetilde W_{\sigma^2 s}.
		\]
		
		Now define \(W_s:=\sigma^{-1}\widetilde W_{\sigma^2 s}\), so that \(W\) is again a Brownian motion by Brownian scaling, and set \(R_s:=s(Y-x)+W_s\). Then, for every \(s\ge0\),
		$
		R_s=\frac{1+\sigma^2 s}{\sigma^2}\bigl(\widetilde X_{\tau_\sigma(s)}^\sigma-x\bigr)
		$. 
		Conversely, if \(t<1\) and \(s=t/(\sigma^2(1-t))\), then
		$
		\widetilde X_t^\sigma=x+\sigma^2(1-t)\,R_{\,t/(\sigma^2(1-t))}
		$. 
		Thus the paths \((R_r)_{0\le r\le s}\) and \((\widetilde X_u^\sigma)_{0\le u\le \tau_\sigma(s)}\) determine each other deterministically, and hence generate the same \(\sigma\)-field. Therefore
		\[
		M_{\tau_\sigma(s)}^\sigma
		=
		\E\big[Y\mid (\widetilde X_u^\sigma)_{0\le u\le \tau_\sigma(s)}\big]
		=
		\E\big[Y\mid (R_r)_{0\le r\le s}\big]
		=:Z_s
		\]
		for every \(s\ge0\). In particular,  we have the pathwise identity
		$
		(M^\sigma_{\tau_\sigma(s)})_{s\ge0}=(Z_s)_{s\ge0},\,\text{a.s.} 
		$ \qedhere
	\end{proof}
	
	Thus all F\"ollmer martingales associated with the same coupling \(m\), but different reference volatilities, are in fact one and the same process written in different time scales. The theorem also suggests that the canonical time scale for the F\"ollmer martingale is not the original interval \([0,1]\), but the information time \(s\in[0,\infty)\) of the filtering problem.
	Indeed, since
	$\frac{R_s}{s}=(Y-X)+\frac{W_s}{s},$
	time \(s\) corresponds to observing \(Y-X\) through Gaussian noise of covariance \(s^{-1}I_d\).
	In this sense, \(1/s\) is the effective observation noise variance, and \(s\) is the natural parameter governing the decay of the posterior covariance.

       { We next comment on a certain consistency or self-similarity of sorts for Föllmer martingales:
        \begin{remark}
Let $M$ be the Föllmer martingale relative to $B$, fix $\tau\in(0,1)$ deterministic and set $\mu_\tau:=\Law(M_\tau)$. By Theorem \ref{thm:intro_scaling_filter}, with $s=\tau/(1-\tau)$, Bayes' formula and
Girsanov theorem give
\[
\Law(M_1\mid M_\tau=z)(dy)
=
\exp\bigl(\psi_\tau(y)+h_\tau(z)\cdot y- A_\tau(h_\tau(z))\bigr)\nu(dy),
\]
where
\[
\psi_\tau(y):=\psi(y)-\frac{\tau}{2(1-\tau)}|y|^2,\qquad
A_\tau(\eta):=\log\int e^{\psi_\tau(\tilde{y})+\eta\cdot \tilde{y}}\nu(d\tilde{y}),
\]
and $h_\tau=\nabla A_\tau$. Equivalently,
$\Law(M_\tau,M_1)$ has the Gibbs form \eqref{eq:introGibbs} with marginals $(\mu_\tau,\nu)$.
Hence $\Law(M_\tau,M_1)$ is the martingale Schrödinger bridge from $\mu_\tau$ to
$\nu$, and the restarted process
$
\bar M_u:=M_{\tau+u(1-\tau)},\; u\in[0,1],
$
is the Föllmer martingale from $\mu_\tau$ to $\nu$. Similar statements and formulae hold for $\Law(M_\tau,M_{\tau '})$. 
\end{remark}
}
As a corollary of this remark, we conclude that for Föllmer martingales, the full infinite-dimensional posterior law
		\[
		\Pi_s:=\Law\big(Y\mid X,(R_r)_{0\le r\le s}\big),
		\]
is already determined by the finite-dimensional estimator $ 
		Z_s=\E[Y\mid X,(R_r)_{0\le r\le s}],
		$ 
together with $s$ and the updated explicit expressions for $(\varphi,\psi,h)$. Hence, $\Pi_s$ evolves within an exponential family with finite-dimensional parameter space $\R_+\times \R^d$. We leave it as an open question, whether these properties already characterize Föllmer martingales.

	We close this part with a particular, one-dimensional case, where the Föllmer martingale admits a particularly explicit form:

	\begin{remark}
		\label{prop:bernoulli_follmer_wf}
		Let \(\mu=\delta_{1/2}\) and \(\nu=\frac12\delta_0+\frac12\delta_1\), and let \(M=(M_t)_{t\in[0,1]}\) be the corresponding F\"ollmer martingale.
		Set
		$
		Z_s:=M_{s/(1+s)}, \,s\ge 0
		$. 
		Then there exists a Bernoulli random variable \(Y\) with \(\P(Y=1)=\P(Y=0)=\frac12\) and an independent Brownian motion \(W\) such that, with
		\[
		R_s:=s\Bigl(Y-\frac12\Bigr)+W_s, \qquad s\ge 0,
		\]
		one has
		\[
		Z_s=\E[Y\mid \F_s^R]=\P(Y=1\mid \F_s^R)=\frac{1}{1+e^{-R_s}}, \qquad s\ge 0,
		\]
		where \(\F_s^R:=\sigma(R_r:0\le r\le s)\).
		Moreover, \(Z\) is the unique weak solution of the filtering SDE
		\[
		dZ_s=Z_s(1-Z_s)\,dB_s, \qquad Z_0=\frac12.
		\]
		This is a particular case of the Wonham filter \cite{Wo64}, but the reader can find a self-contained argument in \cite[Proposition 5.2]{BaBeBiLe26v1}. Equivalently, the Föllmer process from $\mu$ to $\nu$ is a mixture of two Brownian bridges, and using Bayes formula and the Girsanov theorem one readily arrives at this SDE for $Z$.
	\end{remark}

	\section{Examples}\label{sec:examples}
	
	
	
	This section presents two examples of the martingale Schr\"odinger problem
	\eqref{eq:introP} and its continuous-time counterpart \eqref{eq:introP_cont},
	and compares them with their analogue Bass-type problems
	\begin{align}
		\label{eq:introP_bass}\tag{P$_B$}
		\inf_{\pi\in \MT(\mu,\nu)}
		\int \mathcal W_2^2(\pi_x,\gamma)\,\mu(dx),\\
		\label{eq:introP_cont_bass}\tag{P$^{cont}_B$}
		\inf_{\substack{M_0\sim\mu,\;M_1\sim\nu,\\
				M_t=M_0+\int_0^t\sigma_s\,dB_s}}
		\E\!\left[\int_0^1|\sigma_t-I|^2\,dt\right].
	\end{align}
	For our first example, we consider the case of Gaussian marginals, and show that the static optimizers of \eqref{eq:introP} and
	\eqref{eq:introP_bass} are equal, while the continuous-time optimizers of
	\eqref{eq:introP_cont} and \eqref{eq:introP_cont_bass} coincide up to
	change of basis and a coordinatewise time change. We then construct an example with three-point marginals, where the static
	optimizers of \eqref{eq:introP} and \eqref{eq:introP_bass} no longer coincide.
	
	
	\subsection{Gaussian Marginals}
	Let
	$
	\mu=\mathcal N(b_0,\Sigma_0), \
	\nu=\mathcal N(b_1,\Sigma_1)
	$ with $\Sigma_0,\Sigma_1\succ0$,
	and assume that $\mu\lc\nu$.
	By \cite[Theorem~6]{Mu01},
	this is equivalent to $b_0=b_1$ and $
	\Delta\coloneq \Sigma_1-\Sigma_0\succeq0.$
	As relative entropy is invariant under translation, we can assume w.l.o.g. that $\mu$ and $\nu$ are centred throughout this subsection. Additionally, we will assume that $\Delta\succ0$, which is necessary for
	\eqref{eq:introP} to have finite value. Indeed, if $\Delta$ were singular,
	then for some $v\ne0$ and every $ m \in\MT(\mu,\nu)$,
	\[
	\Cov_\pi\bigl(v^\top(Y-X)\bigr)=v^\top\Delta v=0,
	\]
	so $\pi$ would be supported on a hyperplane. This is incompatible with
	finite entropy, since then $\pi\ll\mu\otimes\nu$ and $\mu\otimes\nu$ is
	absolutely continuous with respect to Lebesgue measure.
	
	Let $(X,Y)\sim m\in\MT(\mu,\nu)$. Since $m$ is a martingale, its mean and covariance matrix are given by
	\begin{equation*}\label{eq:mean_cov_XY}
		b_{XY}=\binom{0}{0},
		\qquad
		\Sigma_{XY}=
		\begin{pmatrix}
			\Sigma_0 & \Sigma_0\\
			\Sigma_0 & \Sigma_1
		\end{pmatrix}.
	\end{equation*}
	Now, define $m^G := \mathcal N(b_{XY},\Sigma_{XY})$: one easily verifies that $m^G\in\MT(\mu,\nu)$. By the above observation, it is the unique Gaussian measure in $\MT(\mu,\nu)$.
	
    As $\Delta\succ0$, by the block determinant formula,
	$
	\det(\Sigma_{XY})
	=
	\det(\Sigma_0)\det(\Sigma_1-\Sigma_0)
	=
	\det(\Sigma_0)\det(\Delta)>0$, which prove $\Sigma_{XY} \succ 0$. We may therefore apply \cite[Theorem~8.6.5]{CoTh06}, which states that among all probability laws on $\R^{2d}$ with prescribed zero mean and nonsingular covariance matrix, the Gaussian law on $\R^{2d}$ with that same mean vector and covariance matrix uniquely minimizes $H(\pi\mid\mu\otimes\nu)$.\footnote{The statement in \cite[Theorem~8.6.5]{CoTh06} is formulated in terms of differential entropy. However, since the marginals $\mu$ and $\nu$ are fixed, maximizing the joint differential entropy is equivalent to minimizing $H(\pi\mid\mu\otimes\nu)$.}
	This proves that $m^G$ is the martingale Schrödinger bridge between $\mu$ and $\nu$. 
    
    To compute the dual maximizer and the base measure $\bar \mu$, observe that $m^G = \Law(X,X+Z)$ where $X\sim\mathcal N(0,\Sigma_0)$, $
	Z\sim\mathcal N(0,\Delta)$, and $X$ and $Z$ are independent. Thus
	\begin{equation}\label{eq:disintegration_gaussian}
	    m^G(dx,dy)=\mu(dx)\,\gamma_x^\Delta(dy),
	\end{equation}
	where $\gamma_x^\Sigma$ denotes the Gaussian measure $\mathcal N(x,\Sigma)$ and $\phi_x^\Sigma$ its density. Hence
	\[
	\frac{dm}{d(\mu\otimes\nu)}(x,y)
	=\frac{\phi_x^\Delta(y)}{\phi_0^{\Sigma_1}(y)} = \left(\frac{\det\Sigma_1}{\det\Delta}\right)^{1/2}
	\exp\!\left(
	-\frac12 (y-x)^\top\Delta^{-1}(y-x)
	+\frac12 y^\top\Sigma_1^{-1}y
	\right),
	\]
	and a straightforward computation yields 
	gives
	$
	H(m\mid\mu\otimes\nu)
	=2^{-1}\log ( \det\Sigma_1 (\det\Delta)^{-1}).
	$
	Comparing with the dual form \eqref{eq:density_mart_opt}, we obtain (up to an affine shift):
    $$\bar\varphi(\bar x)=-\frac12 \bar x^\top\Delta \bar x
		+\frac12\log\frac{\det\Sigma_1}{\det\Delta}, \qquad \psi(y)=\frac12 y^\top(\Sigma_1^{-1}-\Delta^{-1})y,$$
    and $h(x)=\Delta^{-1}x$. Immediately, $\bar \mu = h_\# \mu = \mathcal N(0,\Delta^{-1} \Sigma_0 \Delta^{-1}).$ 	
	\begin{proposition}\label{pro:comupation_gaussian_Bass}\hfill
		\begin{enumerate}
			\item The jointly Gaussian martingale coupling $m^G$ is also the minimizer of \eqref{eq:introP_bass}.
			\item The optimizer $M^{\mathrm B}$ of \eqref{eq:introP_cont_bass} admits a stochastic integral representation
			\begin{equation} \label{eq:M^B_t}
				M_t^{\mathrm B}=M_0^{\mathrm B}+\Delta^{1/2}W_t,
				\qquad t\in[0,1],
			\end{equation}
			where $W$ is a standard Brownian motion independent of $M_0^B$. In particular, $M^{\mathrm{B}}$ is Gaussian.
		\end{enumerate}
	\end{proposition}

	\begin{proof}
		Let $(X,Y) \sim m \in\MT(\mu,\nu)$ and $
		S(x)=\Cov_{m_x}(Y)$. 
		Since $m$ is a martingale coupling, $m_x$ has mean $x$ for $\mu$-a.e.\ $x$, and by the law of total covariance,
		\[
		\int S(x)\,\mu(dx)=\Sigma_1-\Sigma_0=\Delta.
		\]
		Using Gelbrich's inequality (\cite[Theorem~2.1]{Ge90}), we have
		\[
		W_2^2(m_x,\gamma)\ge W_2^2(\mathcal N(x,S(x)),\gamma)
		=|x|^2+\tr(S(x))+d-2\tr(S(x)^{1/2}).
		\]
		Integrating and using Jensen's inequality together with the concavity of the map $A\mapsto \tr(A^{1/2})$ on positive semidefinite matrices (see \cite[Theorem~4.2.3]{Bh07}) and Equation \eqref{eq:disintegration_gaussian}, we obtain 
		\[
		\int W_2^2(m_x,\gamma)\,\mu(dx)
		\ge
		\int \Bigl(|x|^2+\tr(\Delta)+d-2\tr(\Delta^{1/2})\Bigr)\,\mu(dx)=\int W_2^2(m_x^G,\gamma)\,\mu(dx).
		\]
		Thus, $m^G$ is also the minimizer of \eqref{eq:introP_bass}.
		
		We now construct the optimizer of \eqref{eq:introP_cont_bass} using the Bass construction, see \cite[Remark~2.3]{BaBeHuKa20}. For each $x\in\R^d$, define
		\[
		T_x(z):=x+\Delta^{1/2}z,
		\qquad z\in\R^d.
		\]
		Then $T_x$ is the gradient of a convex function such that $(T_x)_\#\gamma=m_x^G$.
		Let $W=(W_t)_{t\in[0,1]}$ be a standard Brownian motion started at $0$. For each $x\in\R^d$, define
		\[
		M_t^x:=\E[T_x(W_1)\mid W_t].
		\]
		Since $T_x$ is affine, we obtain
		\[
		M_t^x
		=
		\E[x+\Delta^{1/2}W_1\mid W_t]
		=
		x+\Delta^{1/2}W_t,
		\qquad t\in[0,1].
		\]
		Now let $X\sim\mu$ be independent of $W$, and define
		\[
		M^{\mathrm{B}}_t:=M_t^X=X+\Delta^{1/2}W_t,
		\qquad t\in[0,1].
		\]
		Then $M^{\mathrm{B}}$ is a Gaussian martingale, with
		$M^{\mathrm{B}}_0=X\sim\mu$. Since $m^G$ is the unique optimizer of \eqref{eq:introP_bass}, \cite[Theorem~2.2]{BaBeHuKa20} implies that $M^{\mathrm{B}}$ is the unique-in-law optimizer of \eqref{eq:introP_cont_bass}.
	\end{proof}

	Let $M^{\mathrm E}$ denote the optimizer of \eqref{eq:introP_cont}. We next show that it is Gaussian as well.
	
	\begin{proposition}\label{pro:computation_gauss_Föllmer}
        The Föllmer martingale $M^{\mathrm E}$ admits the stochastic integral representation
		\begin{equation}\label{eq:M^E_t}
			M_t^{\mathrm E}
			=
			M_0^{\mathrm E}
			+\int_0^t \Delta\bigl((1-s)I+s\Delta\bigr)^{-1}\,dW_s,
		\end{equation}
        where $(W_t)_{t \in [0,1]}$ stands for a standard Brownian motion independent of $M_0^{\mathrm E}$. In particular, $M^{\mathrm E}$ is a Gaussian process.
	\end{proposition}
	
	\begin{proof}
		Fix $x\in\R^d$. By Equation \eqref{eq:density_mart_opt}, it holds
		\[
		\frac{dm^G_x}{d\gamma_x}(y) = \frac{d\gamma_x^{\Delta}}{d \gamma_x^I}(y)
		=
		\frac{1}{\sqrt{\det\Delta}}
		\exp\!\left(
		-\frac12 (y-x)^\top(\Delta^{-1}-I)(y-x)
		\right).
		\]
		Therefore, the conditional law $q_{1-t}^{x,z}$ defined in \eqref{eq:def-q-main} is given by
		\[
		q_{1-t}^{x,z}(dy)
		=\frac{1}{Z_{t,x,z}}
		\exp\!\left(
		-\frac{1}{2(1-t)}|y-z|^2
		-\frac12 (y-x)^\top(\Delta^{-1}-I)(y-x)
		\right)\,dy,
		\]
		where $Z_{t,x,z}$ is the normalizing constant.
		Thus $q_{1-t}^{x,z}$ is a Gaussian measure with covariance 
		\[
		\Cov_{q_{1-t}^{x,z}}(Y)
		=
		\left(\Delta^{-1}+\frac{t}{1-t}I\right)^{-1}
		=
		(1-t)\Delta\bigl((1-t)I+t\Delta\bigr)^{-1}.
		\]
		By Equation \eqref{eq:explicit_M_via_X} and Equation \eqref{eq:grad-logH-moment}, the volatility $\sigma_t^x$ of $M^{\mathrm{E}}$ conditioned by $M_0^{\mathrm E} = x$ satisfies 
        \begin{equation*}
        \sigma^{x}_t(z)
		=
		\frac{1}{1-t}\Cov_{q_{1-t}^{x,z}}(Y)
		=
		\Delta\bigl((1-t)I+t\Delta\bigr)^{-1},
        \end{equation*}
        which concludes.
	\end{proof} 
	We now prove that, up to a change of basis and a coordinatewise time change, the processes $M^\mathrm{E}$ and $M^\mathrm{B}$ coincide in our Gaussian setting. In the following, given a $\R^d$-valued process $(X_t)_{t \in [0,1]} = ((X^1_t,\dots,X^d_t))_{t\in [0,1]}$ and a $[0,1]^d$-valued map $\tau = (\tau_1,\dots,\tau_d)$, we write $X_{\tau(s)} := \left(X_{\tau^1(s)}^1,\dots,X_{\tau^d(s)}^d\right)$.
    
	\begin{proposition}
		Let
		$
		\Delta=U\,\diag(\lambda_1,\dots,\lambda_d)\,U^\top
		$ be a spectral decomposition of $\Delta$, and let $\tilde M^{\mathrm E} := U^\top M^{\mathrm E}$ and $\tilde M^{\mathrm B} := U^\top M^{\mathrm B}$ denote the processes $ M^{\mathrm E}$ and $ M^{\mathrm B}$ in the spectral basis.
        Then
        \[
		\left(\tilde M_t^{\mathrm E}\right)_{t\in[0,1]}
		\stackrel{d}{=}
		\left(\tilde M_{\tau(t)}^{\mathrm B}\right)_{t \in[0,1]},
		\]
        where $\tau$ is the time-change defined by
        $$\tau(t) = \left(\frac{t \lambda_1}{1-t+t \lambda_1},\dots, \frac{t \lambda_d}{1-t+t \lambda_d}\right), \qquad t \in [0,1].$$
	\end{proposition}
	
	\begin{proof}
		Set $N_t^B= \tilde M_{\tau(t)}^B$ and
		$N_t^E= \tilde M_t^E$. By Propositions \ref{pro:comupation_gaussian_Bass} and \ref{pro:computation_gauss_Föllmer}, $N^E$ and $N^B$ are both centred Gaussian martingales. Thus, by the martingale property, both mean functions are constant 
        equal to $U^\top \int x \mu(dx)$, and for $i \in \{ B,E \}$ the covariance function $C_i$ of $N^i$ satisfies
        $$C_i(s,t) = \Cov(N_s^i - N_0^i) + \Cov(N_0^i), \qquad \text{for all }  s<t.$$
        As $\Cov(N_0^E)=\Cov(N_0^B) = U^\top \Sigma_0 U$, our statement boils down to prove $\Cov(N_s^B - N_0^B) = \Cov(N_s^E - N_0^E)$. For the left-hand side, as $\Delta^{1/2} = U \diag(\sqrt{\lambda_1},\dots, \sqrt{\lambda_d}) U^\top$ and orthonormal transformations of Brownian motion are Brownian, by Equation \eqref{eq:M^B_t}, it holds $$\Cov(N_t^B- N_0^B) = \Cov\left(\diag(\sqrt{\lambda_1},\dots, \sqrt{\lambda_d}) U^\top W_{\tau(t)}\right) = \diag\!\left(
		\frac{t\lambda_1^2}{1-t+t\lambda_1},
		\dots,
		\frac{t\lambda_d^2}{1-t+t\lambda_d}
		\right).$$
        For the right-hand side, Equation \eqref{eq:M^E_t} together with the Itô isometry and a direct computation yields
        \begin{align*}
            \Cov(N_t^{\mathrm E}-&N_0^{\mathrm E}) = U^\top \left[\int_0^t \Delta^2\bigl((1-s)I+s\Delta\bigr)^{-2} ds\right] U\\
            &=
		U^\top\left[t\,\Delta^2\left((1-t)I+t\Delta\right)^{-1}\right]U
		=
		\diag\!\left(
		\frac{t\lambda_1^2}{1-t+t\lambda_1},
		\dots,
		\frac{t\lambda_d^2}{1-t+t\lambda_d}
		\right),
        \end{align*}
        which concludes.
	\end{proof}
	
	\subsection{Three-Point Discrete Marginals}
	
	We next consider the case where both marginals are supported on three points. In contrast with the Gaussian case, the minimizers of \eqref{eq:introP} and \eqref{eq:introP_bass} no longer coincide. We first parametrize the set $\MT(\mu,\nu)$ for general weights and show that these minimizers are characterized by different systems of equations. We then compute both minimizers numerically for explicit choices of weights, confirming that they are indeed distinct, although remarkably close. Let
	\[
	\mu=p_1\delta_{-1}+q_1\delta_0+r_1\delta_{1},
	\qquad
	\nu=p_2\delta_{-2}+q_2\delta_0+r_2\delta_{2},
	\]
	where
	$
	r_1=1-p_1-q_1,
	\;
	r_2=1-p_2-q_2,
	$
	and assume throughout that all atoms have positive mass. We identify any coupling
	$\pi\in\MT(\mu,\nu)$
	with its $3\times 3$ matrix $(\pi_{x,y})$, whose rows are indexed by
	$x\in\{-1,0,1\}$ and columns by $y\in\{-2,0,2\}$.
	A convenient parametrization is obtained by setting
	\[
	u\coloneq \pi_{-1,-2},
	\qquad
	v\coloneq \pi_{0,-2},
	\qquad
	w\coloneq \pi_{1,-2}=p_2-u-v.
	\]
	Imposing the marginal constraints together with the martingale constraints, one finds that every admissible martingale coupling is necessarily of the form
	\begin{equation}\label{eq:pi_uv}
		\pi(u,v)=
		\begin{pmatrix}
			u & \frac{3p_1}{2}-2u & u-\frac{p_1}{2}\\[4pt]
			v & q_1-2v & v\\[4pt]
			w & \frac{r_1}{2}-2w & w+\frac{r_1}{2}
		\end{pmatrix},
		\qquad
		w=p_2-u-v.
	\end{equation}
	The non-negativity constraints $\pi_{x,y}\ge 0$ are equivalent to
	\[
	\frac{p_1}{2}\le u\le \frac{3p_1}{4},
	\qquad
	0\le v\le \frac{q_1}{2},
	\qquad
	p_2-\frac{r_1}{4}\le u+v\le p_2.
	\]
	We denote by $\mathcal S\subset\R^2$ the corresponding feasible polygon in the $(u,v)$-plane. Hence both optimization problems reduce to minimization over the compact convex set $\mathcal S$.

	We first consider the martingale Schr\"odinger problem \eqref{eq:introP}. Since $\mu$ and $\nu$ have full support, every $\pi\in\MT(\mu,\nu)$ satisfies $\pi\ll\mu\otimes\nu$, and therefore
	$
	H(\pi\mid \mu\otimes\nu)<\infty.
	$
	Minimizing
	\[
	H(\pi\mid \mu\otimes\nu)
	=
	\sum_{x,y}\pi_{x,y}\log\!\Big(\frac{\pi_{x,y}}{\mu_x\nu_y}\Big)
	\]
	over $\MT(\mu,\nu)$ is equivalent to minimizing
	$
	J(\pi)\coloneq \sum_{x,y}\pi_{x,y}\log \pi_{x,y}.
	$
	Accordingly, using the parametrization $\pi=\pi(u,v)$ from \eqref{eq:pi_uv}, we define
	$
	F(u,v)\coloneq J(\pi(u,v)),
	\; (u,v)\in\mathcal S.
	$
	Since $a\mapsto a\log a$ is strictly convex on $(0,\infty)$ and $(u,v)\mapsto \pi(u,v)$ is affine, the function $F$ is strictly convex on $\mathcal S$. Therefore the entropy minimizer is unique.
	
	Under the irreducibility assumptions on $(\mu,\nu)$, 
	the optimizer $m^E$ admits the exponential representation
	\eqref{eq:density_mart_opt}, and in particular all its entries are strictly positive. Equivalently, if
	$
	m^E=\pi(u^E,v^E),
	$
	then $(u^E,v^E)$ lies in the interior of $\mathcal S$, so it is characterized by the first-order condition
	$
	\nabla F(u^E,v^E)=0.
	$
	A direct computation yields the system
	\begin{equation}\label{eq:entropy_system}
		\left\{
		\begin{aligned}
			u(2u-p_1)\,(r_1-4w)^2
			&=
			(3p_1-4u)^2\,w\,(r_1+2w),\\[4pt]
			v^2\,(r_1-4w)^2
			&=
			2\,(q_1-2v)^2\,w\,(r_1+2w),
		\end{aligned}
		\right.
		\qquad w=p_2-u-v.
	\end{equation}
	
	We next turn to the Bass problem \eqref{eq:introP_bass}. Denoting its optimizer by
	$
	m^B=\pi(u^B,v^B),
	$
	the corresponding first-order conditions are
	\begin{equation}\label{eq:bass_system}
		\left\{
		\begin{aligned}
			\Phi^{-1}\!\Big(\frac{u}{p_1}\Big)-\Phi^{-1}\!\Big(\frac32-\frac{u}{p_1}\Big)
			&=
			\Phi^{-1}\!\Big(\frac{w}{r_1}\Big)-\Phi^{-1}\!\Big(\frac12-\frac{w}{r_1}\Big),\\[6pt]
			\Phi^{-1}\!\Big(\frac{v}{q_1}\Big)-\Phi^{-1}\!\Big(1-\frac{v}{q_1}\Big)
			&=
			\Phi^{-1}\!\Big(\frac{w}{r_1}\Big)-\Phi^{-1}\!\Big(\frac12-\frac{w}{r_1}\Big),
		\end{aligned}
		\right.
		\qquad w=p_2-u-v,
	\end{equation}
	where $\Phi^{-1}$ denotes the standard normal quantile function.
	
	Thus the two optimizers are characterized by the systems \eqref{eq:entropy_system} and \eqref{eq:bass_system}. We now illustrate the comparison on a concrete example, with parameters chosen so as to accentuate the discrepancy between the two optimizers. The outcome is that, even after such a tuning, the two optimizers remain remarkably close. 	For
	\[
	\mu=0.40\,\delta_{-1}+0.46\,\delta_0+0.14\,\delta_1,
	\qquad
	\nu=0.43\,\delta_{-2}+0.27\,\delta_0+0.30\,\delta_2,
	\]
	the corresponding entropy and Bass optimizers are
	\[
	m^E=
	\begin{pmatrix}
		0.25123 & 0.09755 & 0.05123\\
		0.16085 & 0.13831 & 0.16085\\
		0.01793 & 0.03414 & 0.08793
	\end{pmatrix},
	\qquad
	m^B=
	\begin{pmatrix}
		0.25229 & 0.09543 & 0.05229\\
		0.15941 & 0.14117 & 0.15941\\
		0.01830 & 0.03340 & 0.08830
	\end{pmatrix},
	\]
	and
	$
	(u^E-u^B,\;v^E-v^B)
	=
	(-1.06,\;1.43)\times 10^{-3}.
	$
	Hence $m^E\neq m^B$, although the discrepancy between the two optimizers is very small in this example.

	\bibliographystyle{abbrv}
	\bibliography{joint_biblio}

\end{document}